\documentclass[12pt,reqno]{amsart}
\usepackage{amsfonts,amsmath,amscd,amssymb,amsthm,graphicx,float}
\usepackage{enumerate,hyperref,perpage,url}

\usepackage[margin=2.2cm, a4paper]{geometry}

\theoremstyle{plain}
\newtheorem{thm}{\textbf{Theorem}}[section]

\newtheorem{lemma}[thm]{\textbf{Lemma}}
\newtheorem{prop}[thm]{\textbf{Proposition}}

\theoremstyle{remark}
\newtheorem{defn}[thm]{\textbf{Definition}}
\newtheorem*{ex}{\textbf{Example}}
\newtheorem*{rmk}{\textbf{Remark}}

\numberwithin{equation}{section}

\MakePerPage{footnote} \allowdisplaybreaks 
\newcommand{\norm}[1]{\left|\!\left|{#1}\right|\!\right|}

\newcommand{\const}{\mathrm{const}}

\newcommand\dist{\mathrm{dist}}

\newcommand{\ges}{\gtrsim}
\newcommand{\les}{\lesssim} 
\newcommand\Id{\mathrm{Id}}
\newcommand\M{\mathbb{M}}
\newcommand\N{\mathbb{N}}
\newcommand\ol{\overline}
\newcommand{\Op}{\mathrm{Op}}
\newcommand\R{\mathbb{R}}
\newcommand\s{\mathbb{S}}
\newcommand\SDO{\text{$\Psi$DO}}
\newcommand\SH{\mathbb{SH}}
\newcommand\supp{\mathrm{supp}\,}

\newcommand\WF{\mathrm{WF}}
\newcommand\ve{\varepsilon}
\newcommand\vp{\varphi}

\title[$L^p$ bilinear quasimode estimates]{$L^p$ bilinear quasimode estimates}

\author{Zihua Guo}
\address{School of Mathematical Sciences, Monash University, Clayton, VIC 3800, Australia}
\email{Zihua.Guo@monash.edu}

\author{Xiaolong Han}
\address{Department of Mathematics, California State University Northridge, California 91330, USA}
\email{xiaolong.han@csun.edu}

\author{Melissa Tacy}
\address{Department of Mathematics and Statistics, University of Otago, Otago 9054, New Zealand}
\email{mtacy@maths.otago.ac.nz}

\subjclass[2010]{35P15, 58J40, 35B30, 33C55}

\keywords{Laplacian, bilinear estimates, eigenfunctions, quasimodes}

\begin{document}
\maketitle

\begin{abstract}
In this paper, we investigate the $L^p$ bilinear quasimode estimates on compact Riemannian manifolds. We obtain results in the full range $p\ge2$ on all $n$-dimensional manifolds with $n\ge2$. This in particular implies the $L^p$ bilinear eigenfunction estimates. We further show that all of these estimates are sharp by constructing various quasimodes and eigenfunctions that saturate our estimates.
\end{abstract}

\section{Introduction}
Let $(\M,g)$ be a smooth and compact Riemannian manifold without boundary. We denote $\Delta=\Delta_g$ the Laplace-Beltrami operator on $\M$. An eigenfunction $u$ of $-\Delta$ satisfies $-\Delta u=\lambda^2u$ with $\lambda$ its eigenfrequency. In 1988 Sogge \cite{So1, So2} proved that
\begin{equation}\label{eq:Sogge}
\|u\|_{L^p}\le C\lambda^{\delta(n,p)}\|u\|_{L^2},\quad\text{where }\delta(n,p)=
\begin{cases}
\frac{n-1}{4}-\frac{n-1}{2p} & \text{for }2\le p\le\frac{2(n+1)}{n-1},\\
\frac{n-1}{2}-\frac{n}{p} & \text{for }\frac{2(n+1)}{n-1}\le p\le\infty.
\end{cases}
\end{equation}
Here, $C$ is independent of $\lambda$. In 2007 Koch, Tataru, and Zworski \cite{KTZ} extended this result to quasimodes, i.e. approximate eigenfunctions in the sense that
$$\|(-\Delta-\lambda^2)u\|_{L^2}\le C\lambda\|u\|_{L^2}.$$
In fact, their result holds for Laplace-like semiclassical pseudodifferential operators. 

In this paper, we investigate  bilinear eigenfunction estimates. That is, for two eigenfunctions $u$ and $v$, we estimate $\|uv\|_{L^p}$  in terms of their eigenfrequencies. One can of course use H\"older's inequality and Sogge's $L^p$ \textit{linear} eigenfunction estimates \eqref{eq:Sogge} to prove a upper bound of $\|uv\|_{L^p}$. For example, on a Riemannian surface (i.e. two dimensional Riemannian manifold), let $u$ and $v$ be two $L^2$-normalized eigenfunctions with eigenvalues $\lambda^2\le\mu^2$. Then
$$\|uv\|_{L^2}\le\|u\|_{L^4}\|\|v\|_{L^4}\lesssim\lambda^\frac18\mu^\frac18\le\mu^\frac14,$$
or with a different pair of H\"older indices (among other possible choices)
$$\|uv\|_{L^2}\le\|u\|_{L^\infty}\|v\|_{L^2}\lesssim\lambda^\frac12.$$
The second bound $\lambda^\frac12$ does not depend on the higher frequency $\mu$, but is not necessarily better than the first bound $\mu^\frac14$ (given, say, $\lambda\approx\mu$). However in 2005, Burq, G\'erard, and Tzvetkov \cite{BGT4} proved that
\begin{equation}\label{eq:BGT}
\|uv\|_{L^2}\lesssim\lambda^\frac14.
\end{equation}
This estimate is clearly better than both the two previous bounds. Moreover, they showed that $\lambda^\frac14$ bound is \textit{sharp} on $\s^2$ (see more discussion on the sharpness of bilinear eigenfunction estimates in Section \ref{sec:sharpness}).

This improvement is crucial in Burq, G\'erard, and Tzvetkov's investigation of nonlinear dispersive equations on manifolds \cite{BGT4, BGT5}. They used \eqref{eq:BGT} to study the well-posedness of Cauchy problems involving nonlinear Sch\"odinger equations on compact Riemannian surfaces. In particular, they obtained the \textit{critical} well-posedness regularity for cubic Sch\"odinger equations on $\s^2$.

In fact, Burq, G\'erard, and Tzvetkov \cite{BGT4} proved \eqref{eq:BGT} for spectral clusters. Subsequently, they \cite{BGT2, BGT5} generalised \eqref{eq:BGT} to higher dimensions as follows.

\begin{thm}[$L^2$ bilinear spectral cluster estimates]\label{thm:BGT}
Let $(\M,g)$ be an $n$-dimensional compact manifold and $\chi\in C^\infty_0(\R)$. Write $\chi_\lambda=\chi\left(\sqrt{-\Delta}-\lambda\right)$. Then for all $1\le\lambda\le\mu$ and $\|f\|_{L^2}=\|g\|_{L^2}=1$, we have
$$\left\|\chi_\lambda(f)\chi_\mu(g)\right\|_{L^2}\lesssim
\begin{cases}
\lambda^\frac14 & \text{if }n=2;\\
\lambda^\frac12|\log\lambda|^\frac12 & \text{if }n=3;\\
\lambda^\frac{n-2}{2} & \text{if }n\ge4.
\end{cases}$$
Moreover, all these estimates are sharp (modulo the $\log$ loss in dimension three).
\end{thm}

Note that if $u$ is an eigenfunction with eigenfrequency $\lambda$, then $\chi_\lambda(u)=u$. So Theorem \ref{thm:BGT} in particular applies to eigenfunctions.

The bounds in Theorem \ref{thm:BGT} depend \textit{only} on the lower eigenfrequency $\lambda$. In this paper, we generalise Theorem \ref{thm:BGT} to  $L^p$ bilinear eigenfunction estimates for $p>2$. In this case, the bound depends on  both of the  eigenfrequencies. We are able to derive a full range of sharp (modulo some log loss when $n=3$ and $p=2$) estimates in all dimensions. The two main themes that arise from our results are:
\begin{enumerate}[(1).]
\item The sharp bound for the $L^p$ norm of the product of two eigenfunctions is better than that obtained by
$$\text{H\"older's inequality}+\text{Sogge's $L^p$ eigenfunction estimates};$$
\item In the sharp bound, the higher eigenfrequency has a smaller exponent than the lower eigenfrequency. The extreme examples of this is $L^2$ bilinear estimates, the exponent of the higher frequency is $0$.
\end{enumerate}

From a technical perspective it is just as easy to work with quasimodes (rather than exact eigenfunctions). So in this paper we prove the $L^p$ bilinear estimates for quasimodes. To this end, it is convenient to work in the semiclassical setting. Denote $h=\lambda^{-1}$ the semiclassical parameter. Then an eigenfunction $u$ satisfies $P(h)u(h)=0$, where
$$P(h)=-h^2\Delta-1.$$
We write $p(x,hD)$ as a semiclassical pseudodifferential operator with symbol $p(x,\xi)$ and study $O_{L^{2}}(h)$ quasimodes as given by the following definition.

\begin{defn}[Quasimodes]\label{hquasimodes}
A family $\{u(h)\}$ ($0<h\le h_0\ll1$) is said to be an $O_{L^{2}}(h)$ quasimode of a semiclassical pseudodifferential operator $p(x,hD)$ if
$$\norm{p(x,hD)u(h)}_{L^{2}}\lesssim{}h\norm{u(h)}_{L^{2}}.$$
\end{defn}

In the semiclassical framework, we can treat quasimodes of semiclassical pseudodifferential operators that are similar to $-h^2\Delta-1$ in the same fashion as quasimodes of the Laplacian. We call such operators Laplace-like (see Definition \ref{laplacelike}). Our main theorem states that
\begin{thm}[$L^p$ bilinear quasimode estimates]\label{thm:bilinear}
Assume that $p(x,\xi)$ is a Laplace-like smooth symbol. Let $u(h)$ and $v(\sigma)$ be two families of $O_{L^2}(h)$ and $O_{L^2}(\sigma)$ quasimodes of $p(x,hD)$ and $p(x,\sigma{}D)$, respectively. Suppose that $0<\sigma\le h\le h_0$ for $h_0\ll1$ and both $u(h)$ and $v(\sigma)$ admit localisation property (see Definition \ref{localised}). Then
$$\|u(h)v(\sigma)\|_{L^p}\lesssim{}G_{n,p}(h,\sigma)\norm{u(h)}_{L^{2}}\norm{v(\sigma)}_{L^{2}},$$
where for $n=2$,
$$G_{n,p}(h,\sigma)=
\begin{cases}
h^{-\frac14}\sigma^{\frac{1}{2p}-\frac14}\quad & \text{for }2\le p\le3,\\
h^{\frac{3}{2p}-\frac34}\sigma^{\frac{1}{2p}-\frac14}\quad & \text{for }3\le p\le6,\\
h^{-\frac12}\sigma^{\frac 2p-\frac12}\quad & \text{for }6\le p\le\infty;
\end{cases}$$
for $n\geq{}3$, $(n,p)\neq{}(3,2)$,
$$G_{n,p}(h,\sigma)=\begin{cases}
h^{-\frac{3(n-1)}{4}+\frac{n+1}{2p}}\sigma^{-\frac{n-1}{4}+\frac{n-1}{2p}}\quad&\text{for }2\leq{}p\leq\frac{2(n+1)}{n-1},\\
h^{-\frac{n-1}{2}}\sigma^{-\frac{n-1}{2}+\frac{n}{p}}\quad&\text{for }\frac{2(n+1)}{n-1}\leq{}p\leq\infty;\end{cases}$$
and
$$G_{3,2}(h,\sigma)=h^{-\frac{1}{2}}|\log h|^{\frac12}.$$
Moreover, all these estimates are sharp (modulo the $\log$ loss in the case $(n,p)=(3,2)$).
\end{thm}

We point out that $\chi_\lambda(f)$ considered in Theorem \ref{thm:BGT} is an $O_{L^{2}}(h)$ quasimode of $-h^2\Delta-1$ with $h=\lambda^{-1}$. (See Section \ref{sec:qm}.) Similarly, $\chi_\mu(g)$ is an $O_{L^{2}}(\sigma)$ quasimode of $-\sigma^2\Delta-1$ with $\sigma=\mu^{-1}$. Therefore, the $L^p$ bilinear spectral cluster estimates follow as a simple consequence of the above theorem. Moreover, they are also sharp.

\begin{thm}[$L^p$ bilinear spectral cluster estimates]\label{thm:efnbilinear}
Suppose that $\mu\ge\lambda\ge1$. Then
$$\left\|\chi_\lambda(f)\chi_\mu(g)\right\|_{L^2}\lesssim G_{n,p}(\lambda^{-1},\mu^{-1})\|f\|_{L^2}\|g\|_{L^2},$$
where $G_{n,p}$ is as given in Theorem \ref{thm:bilinear}. In particular, let $u$ and $v$ be two $L^2$-normalized eigenfunctions with eigenvalues $\lambda^2$ and $\mu^2$. Then
$$\|uv\|_{L^p}\lesssim{}G_{n,p}(\lambda^{-1},\mu^{-1}).$$
Moreover, the estimates are sharp on the sphere.
\end{thm}

\subsection*{Application of bilinear eigenfunctions to nonlinear dispersive equations}
An important application of the bilinear eigenfunctions estimates lies in the study of nonlinear dispersive equations on Riemannian manifolds. Consider the Cauchy problem involving cubic (focusing or defocusing) nonlinear Schr\"odinger equation on a compact Riemannian manifold $\M$:
\begin{equation}\label{eq:Schro}
\begin{cases}
i\partial_tu+\Delta u=\pm|u|^2u,\\
u(0,x)=u_0(x).
\end{cases}
\end{equation}
There is a long history in studying the well/ill-posedness of this system, see e.g. Burq, G\'erard, and Tzvetkov \cite{BGT2, BGT5} and the references therein. Denote $H^s(\M)$ the Sobolev space on $\M$. Burq, G\'erard, and Tzvetkov \cite[Theorem 1]{BGT2} proved that \eqref{eq:Schro} is locally well-posed in $H^s$ for $s>1/4$ on the two-sphere $\s^2$. Their argument combines the $L^2$ bilinear eigenfunction estimate \eqref{eq:BGT} with the eigenvalue distribution on $\s^2$. On the questions of ill-posedness, their earlier work \cite{BGT1} showed that \eqref{eq:Schro} is ill-posed in $H^s$ for $s<1/4$. Therefore, $H^{1/4}$ is the critical regularity on $\s^2$. We express this by saying that the critical threshold $s_c(\s^2)=1/4$, following the notation in \cite{BGT2}. 

It is only on special manifolds (such as the sphere) that such critical regularity threshold is known. Another known example is the torus.  Bourgain \cite{Bo} and Burq, G\'erard, and Tzvetkov \cite{BGT1}, showed that  $s_c(\mathbb T^2)=0$  where $\mathbb T^2$ is the two-dimensional torus. It remains \textit{open} to find the critical threshold on general Riemannian manifolds, and to relate it to geometry of the manifold. For more information about nonlinear Schr\"odinger equations on Riemannian manifolds, we refer to \cite{Bo, BGT1, BGT3, BGT4, BGT5} and the references therein.

For the nonlinear systems with other nonlinear terms such as $|u|^{2k}u$, one may use $L^p$ bilinear eigenfunction estimates to establish similar well-posedness. See the recent work of Yang \cite{Y} about such nonlinear Schr\"odinger systems on zoll manifolds.

\subsection*{Related literature}
Burq, G\'erard, and Tzvetkov \cite[Theorem 3]{BGT5} also proved the following $L^2$ trilinear eigenfunction estimates in dimension two and three.

\begin{thm}[$L^2$ trilinear spectral cluster estimates]
Let $(\M,g)$ be an $n$-dimensional compact manifold and $\chi\in C^\infty_0(\R)$. Given $1\le\lambda\le\mu\le\nu$ and $\|e\|_{L^2}=\|f\|_{L^2}=\|g\|_{L^2}=1$, we have
$$\left\|\chi_\lambda(e)\chi_\mu(f)\chi_\nu(g)\right\|_{L^2}\le
\begin{cases}
C\lambda^\frac14\mu^\frac14 & \text{if }n=2;\\
C_\ve\lambda^{1-\ve}\mu^{\frac12-\ve} & \text{if }n=3.
\end{cases}$$
Here, $\ve\in(0,1]$; $C_\ve$ depends on $\ve$ and $\M$; but $C$ and $C_\ve$ are independent of $\lambda$, $\mu$, and $\nu$.
\end{thm}

\begin{rmk}
We point out that when $\mu\approx\lambda$ our $L^3$ bilinear estimate in Theorem \ref{thm:bilinear} can be derived from the above $L^2$ trilinear one. 
\end{rmk}

We remark that the bilinear eigenfunction estimates are sharp on the sphere. However, they are far from optimal in the case of torus. For example, on the two dimensional torus $\mathbb T^2$, given two $L^2$-normalized eigenfunction $u$ and $v$ with eigenvalues $\lambda^2$ and $\mu^2$,
$$\|uv\|_{L^2}\le\|u\|_{L^4}\|v\|_{L^4}\le C,$$
where $C$ is independent of $\lambda$ and $\mu$. This follows the classical result of Zygmund \cite{Zy} that $\|u\|_{L^4}\le C$.

If the sectional curvatures on the manifold are negative everywhere, then the $L^\infty$ eigenfunction estimate can be improved. See B\'erard \cite{Be}. Let $u$ be an $L^2$-normalized eigenfunction with eigenvalue $\lambda^2$. Then
$$\|u\|_{L^\infty}\lesssim\frac{\lambda^\frac{n-1}{2}}{\sqrt{\log\lambda}}.$$
It thus easily follows that the $L^\infty$ bilinear estimate in Theorem \ref{thm:efnbilinear} can be improved by a logarithmical factor:  Let $u$ and $v$ be two $L^2$-normalized eigenfunction with eigenvalues $\lambda^2$ and $\mu^2$. Then
$$\|uv\|_{L^\infty}\lesssim\frac{(\lambda\mu)^\frac{n-1}{2}}{\sqrt{\log\lambda\log\mu}}.$$
It is interesting to see if the $L^p$ bilinear estimate for $p<\infty$ can also be improved on such manifolds.

\subsection*{Organisation}
We organise the paper as follows. In Section \ref{sec:loc}, we use semiclassical analysis and localisation to reduce the problem to a local one. In Section \ref{sec:bilinear}, we prove Theorem \ref{thm:bilinear}. In Section \ref{sec:sharpness}, we show that the $L^p$ bilinear estimates are sharp by constructing appropriate quasimodes and spherical harmonics. Section \ref{sec:SA} is an appendix containing the basic semiclassical analysis on which this paper relies. 

Throughout this paper, $A\lesssim B$ ($A\gtrsim B$) means $A\le cB$ ($A\ge cB$) for some constant $c$ independent of $\lambda$ or $h$; $A\approx B$ means $A\lesssim B$ and $B\lesssim A$; the constants $c$ and $C$ may vary from line to line.

\section{Semiclassical analysis and localisation}\label{sec:loc}
In this section, we use semiclassical analysis to reduce the $L^p$ bilinear quasimode estimates to a localised problem. For the readers' convenience, we include a brief discussion on the background of semiclassical analysis in the appendix (Section \ref{sec:SA}). 

As any two quantisation procedures differ only by an $O_{L^{2}}(h)$ term, we adopt the left quantisation for symbols, that is,
$$p(x,hD)u(x)=\frac{1}{(2\pi{}h)^{n}}\int{}e^{\frac{i}{h}\langle{}x-y,\xi\rangle}p(x,\xi)u(y)\,d\xi dy.$$
We require that both functions $u(h)$ and $v(\sigma)$ are semiclassically localised with respect to the relevant parameter $h$ or $\sigma$, respectively (see Definition \ref{localised}) and that $p(x,\xi)$ is Laplace-like (see Definition \ref{laplacelike}).

\begin{defn}[Semiclassical localisation]\label{localised} 
A tempered family $\{u(h)\}$ is semiclassically localised if there exists a function $\chi\in C^\infty_0(T^*\M)$ such that
$$u(h)=\chi(x,hD)u(h)+O_{L^2}(h^\infty).$$
\end{defn}

An immediate consequence is that if a family $u=u(h)$ is localised then
\begin{equation}\label{eq:locLp}
\|u\|_{L^q}\le Ch^{n\left(\frac1q-\frac1p\right)}\|u\|_{L^p}+O_{L^{2}}(h^\infty),
\end{equation}
where $1\le p\le q\le\infty$ (see Theorem \ref{thm:scLp}). In particular,
$$\norm{u}_{L^{p}}\lesssim{}h^{-\frac{n}{2}+\frac{n}{p}}\norm{u}_{L^{2}}.$$
Note that the eigenfunctions $u_j=u(\lambda_j^{-1})$ admit localisation property. However, the above inequality gives a worse estimate than Sogge's estimates in \eqref{eq:Sogge}.

\begin{defn}[Laplace-like symbols and operators]\label{laplacelike}
We say that a symbol $p(x,\xi)$ (or its associated semiclassical pseudodifferential operator) is Laplace-like if 
\begin{enumerate}
\item for all $(x_{0},\xi_{0})\in T^*\M$ such that $p(x_{0},\xi_{0})=0$, $\nabla_{\xi}p(x_{0},\xi_{0})\neq{}0$;
\item for all $x_0\in\M$, the hypersurface $\{\xi\in T^*_x\M:p(x_{0},\xi)=0\}$ has positive definite second fundamental form.
\end{enumerate}
\end{defn}

Now suppose that  $u(h)$ and $v(\sigma)$ are $L^2$ normalised semiclassically localised quasimodes of $p(x,hD)$ and $p(x,\sigma{}D)$, respectively:
$$p(x,hD)u(h)=O_{L^2}(h),\quad\text{and}\quad p(x,\sigma D)v(\sigma)=O_{L^2}(\sigma),$$
where $0\le \sigma\le h\le h_0$. From the assumption, they both admit the localisation property. We may assume that 
$$u(h)=\chi(x,hD)u(h)+O_{L^2}(h^\infty),\quad\text{and}\quad v(\sigma)=\chi(x,\sigma{}D)v(\sigma)+O_{L^2}(\sigma^\infty),$$
where $\chi\in C^\infty_0(K)$ with $K\Subset T^*\M$. Hence,
\begin{eqnarray*}
\|uv\|_{L^p}&=&\left\|\chi(x,hD)u(h)\chi(x,\sigma{}D)v(\sigma)\right\|_{L^p}+O_{L^{2}}(h^\infty)\|\chi(x,\sigma{}D)v(\sigma)\|_{L^p}\\
&&+O_{L^{2}}(\sigma^\infty)\|\chi(x,hD)u(h)\|_{L^p}+O_{L^{2}}(h^\infty \sigma^\infty)\\
&\le&\left\|\chi(x,hD)u(h)\chi(x,\sigma{}D)v(\sigma)\right\|_{L^p}+O_{L^{2}}(h^\infty)\sigma^{-\delta(n,p)}+O_{L^{2}}(\sigma^\infty)h^{-\delta(n,p)}+O_{L^{2}}(h^\infty\sigma^\infty)\\
&\le&\left\|\chi(x,hD)u(h)\chi(x,\sigma{}D)v(\sigma)\right\|_{L^p}+h^{N}\sigma^{-\delta(n,p)}+\sigma^{N}h^{-\delta(n,p)}.
\end{eqnarray*}
Inspecting the definition of $G_{n,p}(h,\sigma)$ and $\delta(n,p)$, we see that $\sigma^{-\delta(n,p)}\lesssim{}G_{n,p}(h,\sigma)$. Hence, we only need to estimate the first term. Now we further localise each quasimode. Let 
$$\chi=\sum_{i=1}^N\chi_i,$$
where $N<\infty$ and $\chi_i$ has arbitrarily small support. Thus,
\begin{eqnarray*}
\left\|\chi(x,hD)u(h)\chi(x,\sigma{}D)v(\sigma)\right\|_{L^p}&=&\left\|\sum_{i=1}^N\chi_i(x,hD)u(h)\sum_{j=1}^N\chi_j(x,\sigma{}D)v(\sigma)\right\|_{L^p}\\
&\le&\sum_{i,j=1}^N\big\|\chi_i(x,hD)u(h)\cdot\chi_j(x,\sigma{}D)v(\sigma)\big\|_{L^p}.
\end{eqnarray*}
We then reduce the problem to a local one by separately estimating each term 
$$\big\|\chi_i(x,hD)u(h)\cdot\chi_j(x,\sigma{}D)v(\sigma)\big\|_{L^{p}}$$
in the summation. Since semiclassical pseudodifferential operators commute up to an $O(h)$ error (see Section \ref{sec:Elliptic}), we have
$$p(x,hD)\chi_{i}(x,hD)u(h)=\chi_{i}(x,hD)p(x,hD)u(h)+hr(x,hD)u(h).$$
Therefore, if $u(h)$ is an $O_{L^{2}}(h)$ quasimode of $p(x,hD)$, then $\chi_{i}(x,hD)u(h)$ is also an $O_{L^{2}}(h)$ quasimode. 

In each patch, we may assume the Riemannian volume coincides with the Lebesgue density in local coordinates. Since the localisation property is consistent with $O_{L^2}(h)$ quasimodes, we have
$$p(x,hD)\chi_i(x,hD)u(h)=O_{L^2}(h),\quad\text{and}\quad p(x,\sigma{}D)\chi_j(x,\sigma{}D)u(\sigma)=O_{L^2}(\sigma).$$
We may further reduce to studying a product $\chi_{i}(x,hD)u(h)\chi_{j}(x,\sigma{}D)v(\sigma)$ where $p(x,\xi)$ is zero somewhere on $\supp\chi_{i}\cup\supp\chi_{j}$. Suppose this were not the case, then one of the following two cases happens.
\subsection{$|p(x,\xi)|>c$ on $\supp\chi_{i}$}
Since
$$p(x,hD)\chi_{i}(x,hD)u=h{}f,\quad\text{where }\norm{f}_{L^{2}}\lesssim{}\norm{u}_{L^{2}},$$
and $p(x,hD)$ is invertible on the support of $\chi_{i}$, we have
$$\chi_{i}(x,hD)u=hp(x,hD)^{-1}f,$$ 
and so 
$$\norm{\chi_{i}(x,hD)u}_{L^{2}}\lesssim{}h\norm{u}_{L^{2}}.$$
This implies by \eqref{eq:locLp} that
$$\norm{\chi_{i}(x,hD)u}_{L^{\infty}}\lesssim h\cdot h^{-\frac n2}\norm{u}_{L^{2}}$$
because $\chi_{i}(x,hD)u$ is semiclassically localised. Therefore, applying H\"older's inequality and Sogge's $L^{p}$ estimates \eqref{eq:Sogge} on $\chi_j(x,\sigma D)v$, we have
\begin{eqnarray*}
\norm{\chi_{i}(x,hD)u\chi_{j}(x,\sigma{}D)v}_{L^{p}}&\lesssim{}&\norm{\chi_{i}(x,hD)u}_{L^{\infty}}\norm{\chi_{j}(x,\sigma{}D)v}_{L^{p}}\\
&\lesssim{}&h^{-\frac n2+1}\sigma^{-\delta(n,p)}\norm{u}_{L^{2}}\norm{v}_{L^{2}},
\end{eqnarray*}
which is a better estimate than given by Theorem \ref{thm:bilinear}. 
\subsection{$|p(x,\xi)|>c$ on $\supp\chi_{j}$} 
Then
$$\norm{\chi_{j}(x,\sigma{}D)v}_{L^{2}}\lesssim{}\sigma\norm{v}_{L^{2}}.$$
This implies by \eqref{eq:locLp} that
$$\norm{\chi_{j}(x,\sigma D)v}_{L^r}\lesssim\sigma\cdot\sigma^{-n\left(\frac1r-\frac12\right)}\norm{v}_{L^{2}}$$
because $\chi_{j}(x,\sigma D)v$ is semiclassically localised. Therefore, applying H\"older's inequality and Sogge's $L^{p}$ estimates \eqref{eq:Sogge} on $\chi_i(x,hD)u$, we have
\begin{eqnarray*}
\norm{\chi_{i}(x,hD)u\chi_{j}(x,\sigma{}D)v}_{L^{p}}&\lesssim&\|\chi_{i}(x,hD)u\|_{L^\infty}\|\chi_j(x,\sigma{}D)v\|_{L^p}\\
&\lesssim&h^{-\frac{n-1}{2}}\sigma^{-n\left(\frac1p-\frac12\right)+1}\norm{u}_{L^{2}}\norm{v}_{L^{2}},
\end{eqnarray*}
which is also better than that of Theorem \ref{thm:bilinear}. 

Therefore, we assume that there is an $(x_{0},\xi_{0})\in{}\supp\chi_{i}$ and an $(\tilde{x}_{0},\tilde{\xi}_{0})\in{}\supp\chi_{j}$ such that
$$p(x_{0},\xi_{0})=0,\quad\text{and}\quad p(\tilde{x}_{0},\tilde{\xi}_{0})=0.$$

\section{Proof of $L^p$ bilinear quasimode estimates}\label{sec:bilinear}
Following the reduction in the previous section, we may now work with the product
$$\chi_{i}(x,hD)u\cdot\chi_{j}(x,\sigma D)v,$$
where $\chi_{i}$ and $\chi_{j}$ have small support, and there is an $(x_{0},\xi_{0})\in{}\supp\chi_{i}$ and an $(\tilde{x}_{0},\tilde{\xi}_{0})\in{}\supp\chi_{j}$ such that
$$p(x_{0},\xi_{0})=0,\quad\text{and}\quad p(\tilde{x}_{0},\tilde{\xi}_{0})=0.$$
In these coordinate patches, we often drop $\chi_i$ and $\chi_j$ and for notational convenience simply write $u$ and $v$. At this point we use our assumptions on $p(x,\xi)$ to factorise the symbol. Symbol factorisation will allow us to write
$$p(x,\xi)=e(x,\xi)(\xi_{i}-a(x,\xi')),\quad{}|e(x,\xi)|>c$$
for some $\xi_{i}$. We associate $x_{i}$ with time $t$ and treat $u$ as an approximate solution to an evolution problem. We then obtain $L^{p}$ estimates by considering them as Strichartz estimates with equal weighting on space and time. This ``variable freezing'' is a well-established technique; see for example Stein \cite{St}, Sogge \cite{So1, So2}, Mockenhaupt-Seeger-Sogge \cite{MSS}, Burq-G\'erard-Tzvetkov \cite{BGT3}, and references therein. In the semiclassical setting Koch-Tataru-Zworski \cite{KTZ} used it to obtain $L^{p}$ estimates for quasimode in a similar setting to that we are considering.   

Recall that the first assumption for $p(x,hD)$ to be Laplace-like is that whenever $p(x_{0},\xi_{0})=0$ the gradient $\nabla_{\xi}p(x_{0},\xi_{0})\neq{}0$. Therefore, there is some $\xi_{k}$ for which $\partial_{\xi_{k}}p(x_{0},\xi_{0})\neq{}0$. In fact, we may choose a $\xi_{k}$ (which we will call $\xi_{1}$) so that if $\xi=(\xi_{1},\xi')$, then
$$\nabla_{\xi'}p(x_{0},\xi_{0})=0.$$
We now divide our problem into two cases. 
\begin{enumerate}
\item[Case 1.]
$$\partial_{\xi_{1}}p(\tilde{x}_{0},\tilde{\xi}_{0})\neq{}0;$$
\item[Case 2.]
$$\partial_{\xi_{1}}p(\tilde{x}_{0},\tilde{\xi}_{0})=0.$$
\end{enumerate}

\noindent\textbf{Case 1}: By the implicit function theorem, we may factorise the symbol $p(x,\xi)$ such that
$$p(x,\xi)=e_{1}(x,\xi)(\xi_{1}-a_{1}(x,\xi'))\quad\text{on }\supp\chi_{i},$$
$$p(x,\xi)=e_{2}(x,\xi)(\xi_{1}-a_{2}(x,\xi'))\quad\text{on }\supp\chi_{j},$$
where both $|e_{1}(x,\xi)|,|e_{2}(x,\xi)|>c>0$. Hence, we may invert $e_{1}(x,hD)$ and $e_{2}(x,\sigma{}D)$ to obtain
$$(hD_{x_{1}}-a_{1}(x,hD_{x'}))u(h)=O_{L^{2}}(h\|u(h)\|_{L^2}),$$
and
$$(\sigma{}D_{x_{1}}-a_{2}(x,\sigma{}D_{x'}))v(\sigma)=O_{L^{2}}(\sigma\|v(\sigma)\|_{L^2}).$$
We will treat this case in Propositions \ref{dsdf} and \ref{dsfs} by setting $x_{1}=t$ and reducing the problem to one involving semiclassical evolution equations. Picking up on that notation we say that in Case 1, $u$ and $v$ propagate in the \textbf{same} direction. Note that 
$$\{\xi\in T^*_x\M:p(x_{0},\xi)=0\}=\{\xi\in T^*_x\M:\xi_{1}=a_{1}(x_{0},\xi')\},$$
and
$$\{\xi\in T^*_x\M:p(\tilde x_{0},\xi)=0\}=\{\xi\in T^*_x\M:\xi_{1}=a_{2}(\tilde x_{0},\xi')\}.$$
Now since
$$\nabla_{\xi'}a_{1}(x_{0},\xi_{0})=0,$$ 
the second fundamental form is given at this point by
$$\frac{\partial^{2}a_{1}}{\partial\xi_{i}'\partial\xi'_{j}},$$
so since $p(x,\xi)$ is Laplace-like the matrix,
$$\frac{\partial^{2}a_{1}}{\partial\xi_{i}'\partial\xi'_{j}}$$
is positive definite. A similar argument gives that 
$$\frac{\partial^{2}a_{2}}{\partial\xi_{i}'\partial\xi'_{j}}$$
is also positive definite. 

\noindent\textbf{Case 2}: Since $\partial_{\xi_{1}}p(\tilde{x}_{0},\tilde{\xi}_{0})=0$ but $\nabla_{\xi}p(\tilde{x}_{0},\tilde{\xi}_{0})\neq{}0$, there is some other $\xi_{k}$ (which we will call $\xi_{2}$) such that
$$\partial_{\xi_{2}}p(\tilde{x}_{0},\tilde{\xi}_{0})\neq{}0.$$
We choose this direction so that if $\xi=(\xi_{1},\xi_{2},\bar{\xi})$, then
$$\nabla_{\bar{\xi}}p(\tilde{x}_{0},\tilde{\xi}_{0})=0.$$
Now we write
$$p(x,\xi)=e_{1}(x,\xi)(\xi_{1}-a_{1}(x,\xi_{2},\bar{\xi}))\quad\text{on }\supp\chi_{i},$$
$$p(x,\xi)=e_{2}(x,\xi)(\xi_{2}-a_{2}(x,\xi_{1},\bar{\xi}))\quad\text{on }\supp\chi_{j}.$$
Again, $|e_{1}(x,\xi)|,|e_{2}(x,\xi)|>c>0$. So we invert $e_{1}(x,hD)$ and $e_{2}(x,\sigma{}D)$ to obtain
$$(hD_{x_{1}}-a_{1}(x,hD_{x_{2}},hD_{\bar{x}}))u=O_{L^{2}}(h\|u(h)\|_{L^2}),$$
and
$$(\sigma{}D_{x_{2}}-a_{2}(x,\sigma{}D_{x_{1}},\sigma{}D_{\bar x}))v=O_{L^{2}}(\sigma\|v(\sigma)\|_{L^2}).$$
We will treat this case in Proposition \ref{dd} by setting $x_{1}=t_{1}$ and $x_{2}=t_{2}$ and reducing the problem to one involving semiclassical evolution equations propagating in \textbf{different} directions. We adopt the notation $x=(x_{1},x_{2},\bar{x})$ and $\xi=(\xi_1,\xi_2,\bar\xi)$. Note that by choosing the support small enough we may assume that
$$|(\partial_{\xi_{2}}a_{1},\nabla_{\bar{\xi}}a_{1})|\leq{}\ve,\quad\text{and}\quad|(\partial_{\xi_{1}}a_{2},\nabla_{\bar{\xi}}a_{2})|\leq{}\ve.$$
As in Case 1, we obtain that the matrices of second derivatives $\partial^2_{(\xi_2,\bar\xi)}a_1$ and $\partial_{(\xi_1,\bar\xi)}^2a_2$ are non-degenerate.

We are now in a position to complete the estimates of Theorem \ref{thm:bilinear} by studying the above two cases.

\subsection{Propagating in the same direction (Case 1)}\label{sec:Case1}

We split this into two further cases. Case 1a when $\sigma<ch$ (for some small constant $c$ dependent only on the manifold), in this case the frequencies of the quasimodes are significantly different. Case 1b is of course when $ch\leq{}\sigma\leq{}h$, in this case the frequencies are close enough to be treated as the same.

\begin{prop}[Case 1a]\label{dsdf}
Assume the notations in Case 1:
$$\norm{(hD_{x_{1}}-a_{1}(x,hD_{x'}))u(h)}_{L^{2}}\lesssim{}h\norm{u(h)}_{L^{2}},$$
and
$$\norm{(\sigma{}D_{x_{1}}-a_{2}(x,\sigma{}D_{x'}))v(\sigma)}_{L^{2}}\lesssim{}\sigma\norm{v(\sigma)}_{L^{2}},$$
where $\partial^2_{\xi'\xi'}a_1$ and $\partial_{\xi'\xi'}^2a_2$ are positive definite. Suppose that $\sigma<ch$ for some small $c$ dependent only on the manifold. Then
$$\norm{uv}_{L^{p}}\lesssim{}G_{n,p}(h,\sigma)\norm{u}_{L^{2}}\norm{v}_{L^{2}},$$
where $G_{n,p}(h,\sigma)$ is given in Theorem \ref{thm:bilinear}.
\end{prop}

\begin{proof}
We adopt the notation $t=x_{1}$ and write $x\in\M$ as $x=(t,x')$. Let
$$E_{h}[u]=(hD_{x_{1}}-a_{1}(t,x',hD_{x'}))u(h),$$
and
$$E_{\sigma}[v]=(\sigma{}D_{x_{1}}-a_{1}(t,x',\sigma{}D_{x'}))v(\sigma).$$
Using Duhamel's principle, we have
$$u(t,x')=U_{h}(t,0)u(0,x')+\frac{1}{h}\int_{0}^{t}U_{h}(t-s,s)E_{h}[u]\,ds,$$
and
$$v(t,x')=\widetilde{U}_{\sigma}(t,0)v(0,x')+\frac{1}{\sigma}\int_{0}^{t}\widetilde{U}_{\sigma}(t-s,s)E_{\sigma}[v]\,ds,$$
where
$$\begin{cases}
(hD_{t}-a_{1}(s+t,x',hD_{x'}))U_{h}(t,s)=0,\\
U_{h}(0,s)=\Id,\end{cases}$$
and
$$\begin{cases}
(\sigma{}D_{t}-a_{2}(s+t,x',\sigma{}D_{x'})\widetilde{U}_{\sigma}(t,s)=0,\\
\widetilde{U}_{\sigma}(0,s)=\Id.\end{cases}$$
For simplicity of notation, we write $U_h(t):=U_h(t,0), \widetilde U_\sigma(t):=\widetilde U_\sigma(t,0)$. We may then write $uv$ as 
\begin{multline*}
u(t,x')v(t,x')=U_{h}(t)u(0,x')\widetilde U_{\sigma}(t)v(0,x')+\frac{1}{\sigma}U_{h}(t)u(0,x')\int_{0}^{t}\widetilde{U}_{\sigma}(t-s,s){}E_{\sigma}[v]\,ds+\\
\frac{1}{h}\widetilde{U}_{\sigma}(t)v(0,x')\int_{0}^{t}U_h(t-s,s)E_{h}[u]\,ds+\frac{1}{h\sigma}\int_{0}^{t}\int_{0}^{t}U_{h}(t-s_1,s_{1})E_{h}[u]\widetilde{U}_{\sigma}(t-s_2,s_{2})E_{\sigma}[v]\,ds_{1}ds_{2}.
\end{multline*}
Since $t$ is localised and 
$$\norm{E_\sigma[v]}\lesssim\sigma,\quad\norm{E_h[v]}_{L^{2}}\lesssim{}h,$$
it suffices to prove that
\begin{align}\label{eq:bilinear22}
\norm{U_{h}(t)f(x')\widetilde{U}_{\sigma}(t)g(x')}_{L^{p}_{t,x'}}\lesssim{}G_{n,p}(h,\sigma)\norm{f}_{L^{2}_{x'}}\norm{g}_{L^{2}_{x'}}.
\end{align}
Using the parametrix construction in Section \ref{sec:evo}, we have
$$U_{h}(t)f(x')=\frac{1}{h^{n-1}}\int{}e^{\frac{i}{h}(\phi_{1}(t,x',\xi')-\langle{}y',\xi'\rangle)}b_{1}(t,x',\xi')f(y')\,dy'd\xi',$$
and
$$\widetilde{U}_{\sigma}(t)g(x')=\frac{1}{\sigma^{n-1}}\int{}e^{\frac{i}{\sigma}(\phi_{2}(t,x',\eta')-\langle{}w',\eta'\rangle)}b_{2}(t,x',\eta')g(w')\,dw'd\eta',$$
where
\begin{equation}\label{eq:phih}
\partial_{t}\phi_{1}(t,x',\xi')-a_{1}(t,x',\nabla_{x'}\phi_{1})=0,\quad{}\phi_{1}(0,x',\xi')=\langle{}x',\xi'\rangle,
\end{equation}
and
\begin{equation}\label{eq:phisigma}
\partial_{t}\phi_{2}(t,x',\eta')-a_{2}(t,x',\nabla_{x'}\phi_{2})=0,\quad{}\phi_{2}(0,x',\eta')=\langle{}x',\eta'\rangle.
\end{equation}
Therefore, we need to study the $L^{2}(\R^{2n-2})\to{}L^{p}(\R^{n})$ mapping properties of the operator $W_{h,\sigma}(t)$ given by
$$W_{h,\sigma}(t)=\frac{1}{(h\sigma)^{n-1}}\int{}e^{i\big[\frac{1}{h}\phi_{1}(t,x',\xi')-\frac{1}{h}\langle{}y',\xi'\rangle+\frac{1}{\sigma}\phi_{2}(t,x',\eta')-\frac{1}{\sigma}\langle{}w',\eta\rangle\big]}b(t,x',\xi',\eta')q(y',w')\,d\eta'{}d\xi'{}dy'dw',$$
where
$$b(t,x',\xi',\eta')=b_{1}(t,x',\xi')b_{2}(t,x',\eta').$$
We will view this as a Strichartz type estimate with equal weighting on space and time. For unitary operators $V(t)$, such estimates depend only on the dispersive bounds
$$\norm{V(t)V^{\star}(s)}_{L^{1}\to{}L^{\infty}}\lesssim{}|t-s|^{-\kappa}.$$
See for example the abstract framework given in Keel-Tao \cite{KT}. Our operator $W_{h,\sigma}(t)$ is not unitary but we may still use this framework as in Tacy \cite{T} by proving $L^{2}\to{}L^{2}$ estimates to replace unitarity. We prove
$$\norm{W_{h,\sigma}(t)W_{h,\sigma}^{\star}(s)}_{L^{1}_{x'}\to{}L^{\infty}_{x'}}\lesssim{}h^{-\alpha_{\infty}}\sigma^{-\beta_{\infty}}(h+|t-s|)^{-\gamma_{\infty}}(\sigma+|t-s|)^{-\kappa_{\infty}},$$
and
$$\norm{W_{h,\sigma}(t)W_{h,\sigma}^{\star}(s)}_{L^{2}_{x'}\to{}L^{2}_{x'}}\lesssim{}h^{-\alpha_{2}}\sigma^{-\beta_{2}}(h+|t-s|)^{-\gamma_{2}}(\sigma+|t-s|)^{-\kappa_{2}}.$$
Then by interpolating between then obtain
$$\norm{W_{h,\sigma}(t)W_{h,\sigma}^{\star}(s)}_{L^{p'}_{x'}\to{}L^{p}_{x'}}\lesssim{}h^{-\alpha_{p}}\sigma^{-\beta_{p}}(h+|t-s|)^{-\gamma_{p}}(\sigma+|t-s|)^{-\kappa_{p}}.$$
As in the Keel-Tao \cite{KT} framework and in earlier work by Mockenhaupt-Seeger-Sogge \cite{MSS}, and Sogge \cite{So1}, we use Young and Hardy-Littlewood-Sobolev inequalities to resolve the $|t-s|$ integral. 
Therefore, we need to calculate $W_{h,\sigma}(t)W_{h,\sigma}^{\star}(s)$, where
$$W_{h,\sigma}(t)W_{h,\sigma}^{\star}(s)f(x')=\int{}W(t,x',s,z')f(z')dz'.$$
Here,
\begin{multline}\label{Wkernel}
W(t,x',s,z')=\frac{1}{(h\sigma)^{2(n-1)}}\int{}e^{i\psi_{h,\sigma}(t,s,x',y',w',z',\xi'_{1},\xi'_{2},\eta'_{1},\eta'_{2})}b(t,s,x',z',\xi'_{1},\xi'_{2},\eta'_{1},\eta'_{2})\,d\Lambda\\
d\Lambda=dy'dw'd\xi_{1}'d\xi_{2}'{}d\eta_{1}'d\eta_{2}',\end{multline}
where
\begin{eqnarray*}
\psi_{h,\sigma}(t,s,x',y',w',z',\xi'_{1},\xi'_{2},\eta'_{1},\eta'_{2})&=&\frac{1}{h}\Big[\phi_{1}(t,x',\xi_{1}')-\phi_{1}(s,z',\xi'_{2})-\langle{}y',\xi'_{1}-\xi_{2}'\rangle\Big]\\
&&+\frac{1}{\sigma}\Big[\phi_{2}(t,x',\eta_{1}')-\phi_{1}(s,z',\eta'_{2})-\langle{}y',\eta'_{1}-\eta'_{2}\rangle\Big].
\end{eqnarray*}
We use stationary phase asymptotics in Zworski \cite[Theorem 3.16]{Zw} to calculate the $(y',\xi'_{1},w',\eta'_{1})$ integral in \eqref{Wkernel}. As the stationary point is always non-degenerate, we obtain
\begin{eqnarray*}
W(t,x',s,z')&=&\frac{1}{h^{n-1}}\int e^{\frac{i}{h}[\phi_{1}(t,x',\xi')-\phi_{1}(s,z',\xi')]}\tilde b_1(t,s,x',z',\xi')\,d\xi'\\
&&\times\frac{1}{\sigma^{n-1}}\int e^{\frac{i}{\sigma}[\phi_{2}(t,x',\eta')-\phi_{2}(s,z',\eta')]}\tilde b_2(t,s,x',z',\eta')\,d\eta'\\
&:=&K_h (t,x',s,z')\cdot K_\sigma(t,x',s,z').
\end{eqnarray*}
We analyse $K_{h}$ and $K_{\sigma}$ by considering critical points of 
$$\phi_{1}(t,x',\xi')-\phi_{1}(s,z',\xi')\quad\mbox{in $\xi'$}$$
and
$$\phi_{2}(t,x',\eta')-\phi_{2}(s,z',\eta')\quad\mbox{in $\eta'$}.$$
To that end, following Keel-Tao \cite{KTZ}, we study the oscillatory integral
\begin{equation}\label{eq:Kh}
K_h(t,x',s,z')=\frac{1}{h^{n-1}}\int e^{\frac{i}{h}\psi_1(t,s,x',z',\xi')}\tilde b_1(t,s,x',z',\xi')\,d\xi',
\end{equation}
where
$$\psi_1(t,s,x',z',\xi')=\phi_{1}(t,x',\xi')-\phi_{1}(s,z',\xi').$$
The critical point of $\psi_1$ solves
\begin{equation}\label{stationaryh}
\nabla_{\xi'}\psi_1(t,s,x',z',\xi')=\nabla_{\xi'}\Big[\phi_{1}(t,x',\xi')-\phi_{1}(s,z',\xi')\Big]=0.
\end{equation}
From \eqref{eq:phih}, we have that
$$\phi_{1}(t,x',\xi')-\phi_{1}(s,z',\xi')=\langle{}x'-z',\xi'+O(|s|)\rangle+(t-s)(a_{1}(s,z',\xi')+O(|t-s|)).$$ 
Then \eqref{stationaryh} implies that
$$0=(1+O(|s|))\cdot{}(x'-z')+(t-s)(\partial_{\xi'}a_{1}+O(|t-s|)),$$
and thus for critical points to occur we must have
$$|x'-z'|=O(|t-s|).$$
If $|x'-z'|\gg|t-s|$, using integration by parts to \eqref{eq:Kh} we get for any $N\in\N$
$$|K_h(t,x',s,z')|\le\frac{1}{h^{n-1}}\left(1+\frac{|x'-z'|}{h}\right)^{-N}.$$
If $|x'-z'|\lesssim|t-s|$, there might be critical points. The Hessian is given by
\begin{align*}
\partial^{2}_{\xi'\xi'}\psi_1(t,s,x',z',\xi')=&(t-s)(\partial^{2}_{\xi'\xi'}a_{1}+O(|t-s|))+O(|x'-z'|s)\\
=&(t-s)(\partial^{2}_{\xi'\xi'}a_{1}+O(|t|+|s|)).
\end{align*}
Thus by the non-degenerate condition there exists at most one critical point. We know that $\partial_{\xi'\xi'}^{2}a_{1}$ is a non-degenerate matrix, so
\begin{equation}\label{eq:hHessian}
\left|\det\left(\partial^{2}_{\xi'\xi'}\psi_1\right)\right|\geq{}c|t-s|^{n-1}.
\end{equation}
A similar calculation holds for 
\begin{equation}\label{eq:Ksigma}
K_\sigma(t,x',s,z')=\frac{1}{\sigma^{n-1}}\int e^{\frac{i}{\sigma}\psi_2(t,s,x',z',\eta')}\tilde b_2(t,s,x',z',\eta')\,d\eta',
\end{equation}
where
$$\psi_2(t,s,x',z',\eta')=\phi_{2}(t,x',\eta')-\phi_{2}(s,z',\eta').$$
If $|x'-z'|\gg|t-s|$, then
$$|K_\sigma(t,x',s,z')|\le\frac{1}{\sigma^{n-1}}\left(1+\frac{|x'-z'|}{\sigma}\right)^{-N}.$$
If $|x'-z'|\lesssim|t-s|$, there exists at most one critical point that satisfies
\begin{equation}\label{stationarysigma}
\nabla_{\eta'}\psi_2(t,s,x',z',\eta')=\nabla_{\eta'}\Big[\phi_{2}(t,x',\eta')-\phi_{2}(s,z',\eta')\Big]=0.
\end{equation}
The Hessian satisfies 
\begin{equation}\label{eq:sigmaHessian}
\left|\det\left(\partial^{2}_{\eta'\eta'}\psi_2\right)\right|\geq{}c|t-s|^{n-1}.
\end{equation}
We split the kernel into three parts
\begin{itemize}
\item Regime 1. $|t-s|\leq{}C\sigma$,
\item Regime 2. $C\sigma\leq{}|t-s|\leq{}Ch$,
\item Regime 3. $Ch\leq{}|t-s|$,
\end{itemize}
for some $C$ independent of $t,s,h,\sigma$. We write
$$W(t,x',s,z')=W_{1}(t,x',s,z')+W_{2}(t,x',s,z')+W_{3}(t,x',s,z'),$$
where $\chi\in C^\infty_0([-1,1])$ and 
$$W_{1}(t,x',s,z')=\chi(\sigma^{-1}|t-s|)W(t,x',s,z'),$$
$$W_{2}(t,x',s,z')=(1-\chi(\sigma^{-1}|t-s|))\chi(h^{-1}|t-s|)W(t,x',s,z'),$$
$$W_{3}(t,x',s,z')=(1-\chi(\sigma^{-1}|t-s|)-(1-\chi(\sigma^{-1}|t-s|))\chi(h^{-1}|t-s|))W(t,x',s,z').$$
In the following computation, we slightly abuse the notation and use $W_i(t,s)$ as the operator with integral kernel $W_{i}(t,x',s,z')$, $i=1,2,3$.

\subsubsection{Regime 1}
We have that $|t-s|\le C\sigma$ in the support of $W_1$. In this case, note that as there cannot be a critical point in either $\xi'$ or $\eta'$ if $|x'-z'|>C|t-s|$, we can integrate by parts in $\eta'$ to obtain
$$|W_{1}(t,x',s,z')|\leq{}h^{-(n-1)}\sigma^{-(n-1)}\left(1+\frac{|x'-z'|}{\sigma}\right)^{-N}$$
for any natural number $N$. Therefore, applying Young's inequality (see e.g. Sogge \cite[Corollary 2.1.2]{So3}) we obtain that if $W_{1}(t,s)$ is the operator associated with the integral kernel $W_{1}(t,x',s,z')$, then
\begin{equation}\label{eq:W1bound}
\norm{W_{1}(t,s)}_{L^{p'}_{x'}\to{}L^{p}_{x'}}\lesssim{}h^{-(n-1)}\sigma^{-(n-1)+\frac{2(n-1)}{p}}.
\end{equation}

\subsubsection{Regime 2}
We have that $C\sigma\le|t-s|\le Ch$ in the support of $W_2$. If $|x'-z'|\gg|t-s|$ or $K_\sigma$ has no critical point, then we immediately get the same bound \eqref{eq:W1bound} and hence we are done. Now we assume $K_\sigma$ has a unique critical point $\eta'_c$.  By the stationary phase theorem in H\"ormander \cite[Theorem 7.7.5]{Ho}, using \eqref{eq:Ksigma} and \eqref{eq:sigmaHessian}, we have
$$W_2(t,s,x',z')=K_h(t,s,x',z')\cdot\sigma^{-\frac{n-1}{2}}|t-s|^{-\frac{n-1}{2}}e^{\frac{i}{\sigma}\theta(t,s,x',z')}B (t,s,x',z'),$$
where $B(t,s,x',z')$ is bounded and compactly supported and
$$\theta(t,s,x',z')=\psi_2(t,s,x',z',\eta_c')=\phi_{2}(t,x',\eta_c')-\phi_{2}(s,z',\eta_c').$$
Thus we immediately get
\begin{equation}\label{eq:2ndL1Linfty}
\norm{W_{2}(t,s)}_{L^1_{x'}\to L^{\infty}_{x'}}\lesssim{}h^{-(n-1)}\sigma^{-\frac{n-1}{2}}|t-s|^{-\frac{n-1}{2}}.
\end{equation}
To get the $L^{2}_{x'}\to{}L^{2}_{x'}$ estimates we write
$$W_{2}(t,s,x',z')=\int{}W_{2}(t,s,x',z',\xi')\,d\xi',$$
where
$$W_{2}(t,s,x',z',\xi')=h^{-(n-1)}\sigma^{-\frac{n-1}{2}}|t-s|^{-\frac{n-1}{2}}e^{\frac{i}{\sigma}\theta(t,s,x',z')+\frac{1}{h}(\phi_{1}(t,x',\xi')-\phi_{1}(s,z',\xi'))}b(t,x',\xi)B(t,s,x',z').$$
Due to the compact support in $\xi'$ it is enough to estimate the $L^{2}\to{}L^{2}$ norm of $W_{2}(t,s,\xi')$, the operator with kernel $W_{2}(t,s,x',z',\xi')$. 

We will do this by writing
$$\norm{W_{2}(t,s,\xi')f}_{L^{2}_{x'}}^{2}=\int{}\widetilde{W}_{2}(t,s,y',z',\xi')f(z')\overline{f(y')}\,dy'dz',$$
where
\begin{equation}\label{eq:tildeW2}
\widetilde{W}_{2}(t,s,y',z',\xi')=\int{}W_{2}(t,x',s,y',\xi')\overline{W_2(t,x',s,z',\xi')}\,dx'.
\end{equation}
and  estimating $|\widetilde{W}_{2}(t,s,y',z')|$ by integrating by parts in $x'$. To do this we need to know how $B(t,s,x',z')$ behaves under differentiation in $x'$.

If we differentiate \eqref{stationarysigma} in $x'$ gives
$$0=\Id+O(|t|+|s|)+(t-s)\left([\partial_{x'}\eta_{c}]^{T}\partial^{2}_{\eta'\eta'}\phi_2+O(|t-s|)\right).$$
Hence,
\begin{equation}\label{etaderiv}
\partial_{x'}\eta'_{c}(t,s,x',z')=\frac{-1}{t-s}(\Id+O(|t|+|s|))(\partial^{2}_{\eta'\eta'}\phi_2)^{-1},
\end{equation}
and thus,
$$|\partial_{x_{i}}\eta'_c(t,s,x',z')|\leq{}\frac{C}{|t-s|}.$$
We can repeat the process to get
$$|D_{x'}^{\gamma}\eta_c'(t,s,x',z')|\leq{}\left(\frac{C}{|t-s|}\right)^{|\gamma|}$$
for any multi-index $\gamma$. Furthermore, we know from Tacy \cite[Lemma 4.1]{T} that
\begin{equation}\label{eq:dxB}
|D_{x'}^{\gamma}B(t,s,x',z')|\leq{}|D_{x'}^{\gamma}\eta'_{c}(t,s,x',z')|\leq{}\left(\frac{C}{|t-s|}\right)^{|\gamma|}.
\end{equation}
We may now proceed in integrating \eqref{eq:tildeW2} by parts in $x'$.

The phase function is
$$\tilde{\theta}_{\sigma,h}(t,s,x',y',z')=\frac{1}{\sigma}\big[\theta(t,s,x',y')-\theta(t,s,x',z')\big]+\frac{1}{h}\left[\phi_{1}(s,y',\xi')-\phi_{1}(s,z',\xi')\right].$$
Note that the term
$$\frac{1}{h}\left[\phi_{1}(s,y',\xi')-\phi_{1}(s,z',\xi')\right].$$
does not depend on $x'$.  Expanding with a Taylor series about $y'=z'$ we obtain
$$\theta(t,s,x',y')-\theta(t,s,x',z')=\frac{1}{\sigma}\Big[\nabla_{z'}\theta(t,s,x',z')\cdot{}(y'-z')+O(|y'-z'|^{2})\Big].$$
So
$$\nabla_{x'}\tilde{\theta}_{\sigma,h}(t,s,x',y',z')=\frac{1}{\sigma}\Big[\partial^{2}_{x'z'}\theta(t,s,x',z')\cdot{}(y'-z')+O(|y'-z'|^{2})\Big].$$
Therefore, we need an expression for $\partial^{2}_{x'z'}\theta$. We have that
$$\theta(t,s,x',z')=\phi_{2}(t,x',\eta'_{c})-\phi_{2}(t,z',\eta'_{c}),$$
and
\begin{equation}\label{criticalequation}
0\equiv{}\partial_{\eta'}(\phi_{2}(t,x',\eta'_{c})-\phi_{2}(t,x',\eta'_{c})).
\end{equation}
Differentiating \eqref{criticalequation} in $z'$ and $x'$ we obtain
$$\frac{\partial^{2}\theta}{\partial{}x_{i}\partial{}z_{j}}=\sum_{k,l=2}^{n}\frac{\partial^{2}\phi_{2}}{\partial\eta_{l}\partial\eta_{k}}\frac{\partial\eta_c'}{\partial{}x_{i}}\frac{\partial\eta_c'}{\partial{}z_{j}},$$
which in matrix form is
\begin{equation}\frac{\partial^{2}\theta}{\partial{}x'\partial{}z'}=[\partial_{x'}\eta_c']^{T}\partial^{2}_{\eta'\eta'}\phi_2[\partial_{z'}\eta_c'].\label{2ndmatrix}\end{equation}
We already have an expression for $\partial_{x'}\eta_c'(t,s,x',z')$ from \eqref{etaderiv}. A similar process gives us
$$\partial_{z'}\eta_c'(t,s,x',z')=\frac{-1}{t-s}(\Id+O(|t|+|s|))(\partial^{2}_{\eta'\eta'}\phi_2)^{-1}.$$
Combining this with \eqref{2ndmatrix}, we have
\begin{equation}\label{eq:thetapartialpartial}
\frac{\partial^{2}\theta}{\partial{}x'\partial{}z'}=\frac{-1}{t-s}(\partial_{\eta'\eta'}^{2}\phi_2^{-1}+O(|t-s|)).
\end{equation}
Therefore,
$$|\nabla_{x'}\tilde\theta_{\sigma}(t,s,x',y',z')|\geq{}\frac{c|y'-z'|}{\sigma|t-s|}.$$
So each integration by parts of \eqref{eq:tildeW2} with respect to $x'$ gains a factor of 
$$\frac{\sigma|t-s|}{|y'-z'|},$$ 
but, in view of \eqref{eq:dxB}, looses a factor of $|t-s|^{-1}$ from hitting the symbol. Overall each integration by parts gains a factor of
$$\frac{\sigma}{|y'-z'|}.$$
Therefore, we obtain
\begin{eqnarray*}
&&|\widetilde{W}_{2}(t,s,y',z',\xi')|\\
&\lesssim&\!\!\!\!{}h^{-2(n-1)}\sigma^{-(n-1)}|t-s|^{-(n-1)}\!\left(1+\frac{|y'-z'|}{\sigma}\right)^{-N}\!\!\!\!\!\int{}\chi\left(\frac{|y'-x'|}{|t-s|}\right)\chi\left(\frac{|z'-x'|}{|t-s|}\right)\,dx'\\
&\lesssim&\!\!\!\!{}h^{-2(n-1)}\sigma^{-(n-1)}\left(1+\frac{|y'-z'|}{\sigma}\right)^{-N}.
\end{eqnarray*}
So by H\"older and Young inequalities,
$$\norm{W_{2}(t,s,\xi')f}_{L^{2}_{x'}}^{2}\lesssim{}h^{-2(n-1)}\norm{f}_{L^{2}_{x'}}^{2},$$
which implies
$$\norm{W_{2}(t,s,\xi')}_{L^{2}_{x'}\to L^2_{x'}}\lesssim{}h^{-(n-1)}$$
and
\begin{equation}\label{eq:2ndL2L2}
\norm{W_{2}(t,s)}_{L^{2}_{x'}\to L^2_{x'}}\lesssim{}h^{-(n-1)}.
\end{equation}

\subsubsection{Regime 3}
We have that $|t-s|\ge Ch$ in the support of $W_3$. If $|x'-z'|\gg|t-s|$ or $K_\sigma$ has no critical point, then $W_{3}$ has the same bound as $W_{1}$ given by \eqref{eq:W1bound} and we are done. Now we assume $|x'-z'|\les|t-s|$ and $K_\sigma$ has a critical point $\eta'_c$. If $K_h$ has no critical point, then$W_{3}$ has the same bounds as $W_{2}$ given by \eqref{eq:2ndL1Linfty} and \eqref{eq:2ndL2L2} .  So we may also assume $K_h$ also has a critical point, denoted as $\xi'_c$. Therefore, we obtain
$$W_{3}(t,x',s,z')=\frac{1}{(h\sigma)^{\frac{n-1}{2}}}|t-s|^{-\frac{n-1}{2}}|t-s|^{-\frac{n-1}{2}}e^{i\left(\frac{1}{h}\theta_{1}(t,s,x',z')+\frac{1}{\sigma}\theta_{2}(t,s,x',z')\right)}B(t,s,x',z'),$$
where 
$$\theta_{1}(t,s,x',z')=\phi_{1}(t,x',\xi'_{c})-\phi_{1}(t,x',\xi'_{c}),$$
and
$$\theta_{2}(t,s,x',z')=\phi_{2}(t,x',\eta'_{c})-\phi_{2}(t,x',\eta'_{c})$$
for $\xi'_{c}$ and $\eta_{c}'$ the solutions to \eqref{stationaryh} and \eqref{stationarysigma}, respectively. Further we still have
$$|D_{x'z'}^{\gamma}B(t,s,x',z')|\lesssim|t-s|^{-|\gamma|}$$
for any multi-index $\gamma$ thanks to \eqref{eq:dxB}. From this representation, we can directly read off the $L^{1}_{x'}\to{}L^{\infty}_{x'}$ estimates, that is,
\begin{equation}\label{eq:3rdL1Linfty}
\norm{W_{3}(t,s)}_{L^{1}_{x'}\to{}L^{\infty}_{x'}}\lesssim{}h^{-\frac{n-1}{2}}\sigma^{-\frac{n-1}{2}}|t-s|^{-(n-1)}.
\end{equation}
We need only then calculate the $L^{2}_{x'}\to{}L^{2}_{x'}$ norm. As in Regime 2, we have
$$\norm{W_{3}(t,s)f}^{2}_{L^{2}_{x'}}=\int{}\widetilde{W}_{3}(t,s,y',z')f(y')\overline{f(z')}\,dy'dz'.$$
Here,
\begin{eqnarray*}
\widetilde{W}(t,s,z',y)&&\\
&\!\!=\!\!&\!\!\!\!\int W(t,x',s,y')\overline{W(t,x',s,z')}\,dx'\\
&\!\!=\!\!&\!\!\!{}h^{-(n-1)}\sigma^{-(n-1)}|t-s|^{-2(n-1)}\!\!\!\int{}e^{i\left(\frac{1}{h}\tilde{\theta}_{1}(t,s,x',y',z')+\frac{1}{\sigma}\tilde{\theta}_{2}(t,s,x',y',z')\right)}B(t,s,x',z')\overline{B(t,s,x',y')}\,dx',
\end{eqnarray*}
in which
$$\tilde{\theta}_{1}(t,s,x',y',z')=\theta_{1}(t,s,x',y')-\theta_{1}(t,s,x',z'),$$
and
$$\tilde{\theta}_{2}(t,s,x',y',z')=\theta_{2}(t,s,x',y')-\theta_{2}(t,s,x',z').$$
Note that as $\sigma$ is much smaller than $h$ the major oscillatory terms come from this parameter. 
We need to calculate
$$\nabla_{x'}(\tilde{\theta}_{1}(t,s,x',y',z')+\tilde{\theta}_{2}(t,s,x',y',z'))$$
Again we expand around the point $y'=z'$ and obtain
$$\nabla_{x'}(\tilde{\theta}_{1}(t,s,x',y',z')+\tilde{\theta}_{2}(t,s,x',y',z'))=\left(\frac{1}{h}\partial^{2}_{x'z'}\theta_{1}+\frac{1}{\sigma}\partial^{2}_{x'z'}\theta_{2}\right)\cdot(y'-z')+O(|y'-z'|^{2}).$$
We already know from \eqref{eq:thetapartialpartial} that
$$\partial^{2}_{x'z'}\theta_1=\frac{-1}{t-s}(\partial_{\xi'\xi'}^{2}\phi_1^{-1}+O(|t-s|)),\quad\text{and}\quad\partial^{2}_{x'z'}\theta_2=\frac{-1}{t-s}(\partial_{\eta'\eta'}^{2}\phi_2^{-1}+O(|t-s|)).$$
Since $h>C\sigma$, we have
$$\left|\nabla_{x'}(\tilde{\theta_{1}}+\tilde{\theta}_{2})\right|\geq{}\frac{c|y'-z'|}{\sigma|t-s|}.$$
So we may proceed in the same fashion as in the second regime to obtain
$$|\widetilde{W}_{3}(t,s,y',z')|\lesssim{}h^{-(n-1)}\sigma^{-(n-1)}|t-s|^{-2(n-1)}\left(1+\frac{|y'-z'|}{\sigma}\right)^{-N}\int{}\chi\left(\frac{|y'-x'|}{|t-s|}\right)\chi\left(\frac{|z'-x'|}{|t-s|}\right)\,dx'\\.$$
This implies that
$$\norm{W_{3}(t,s)f}_{L^{2}_{x'}}^{2}\lesssim{}h^{-(n-1)}|t-s|^{-(n-1)}\norm{f}_{L^{2}_{x'}}^{2},$$
and
\begin{equation}\label{eq:3rdL2L2}
\norm{W_{3}(t,s)}_{L^{2}_{x'}\to L^{2}_{x'}}\lesssim{}h^{-\frac{n-1}{2}}|t-s|^{-\frac{n-1}{2}}.
\end{equation}

Recall that we needed to prove 
$$\left\|(W_{h,\sigma}(t)q)(x')\right\|_{L^p_{t,x'}}\lesssim G(h,\sigma)\|q(y',w')\|_{L^2_{y',w'}}.$$
By $TT^\star$ argument it suffices to show
$$\norm{\int W_i(t,s)f(s)\,ds}_{L^p_{t,x'}}\le G(h,\sigma)^2\norm{f}_{L^{p'}_{t,x'}},\quad i=1,2,3.$$

\noindent\textbf{Regime 1: $|t-s|\leq{}C\sigma$.}
We get from \eqref{eq:W1bound} and H\"older's inequality that
\begin{eqnarray*}
\left\|\int W_1(t,s)f(s)\,ds\right\|_{L^p_{t,x'}}&\le&\left\|\int_{|t-s|\le C\sigma}h^{-(n-1)}\sigma^{-(n-1)+\frac{2(n-1)}{p}}\norm{f(s,\cdot)}_{L^{p'}_{x'}}\,ds\right\|_{L_t^p}\\
&\le&h^{-(n-1)}\sigma^{-(n-1)+\frac{2n}{p}}\norm{f}_{L^{p'}_{t,x'}}.
\end{eqnarray*}

\noindent\textbf{Regime 2: $C\sigma\leq{}|t-s|\leq{}Ch$.}
Interpolation between \eqref{eq:2ndL1Linfty} and \eqref{eq:2ndL2L2} gives that
$$\norm{W_{2}(t,s)}_{L^{p'}_{x'}\to L^{p}_{x'}}\lesssim{}h^{-(n-1)}\sigma^{-\frac{n-1}{2}+\frac{n-1}{p}}|t-s|^{-\frac{n-1}{2}+\frac{n-1}{p}}.$$
Hence, from H\"older's inequality, we have that
\begin{eqnarray*}
\norm{\int W_2(t,s)f(s)ds}_{L^p_{t,x'}}&\les& \norm{\int_{\sigma\les |t-s|\les h} h^{-(n-1)}\sigma^{-\frac{n-1}{2}+\frac{n-1}{p}} |t-s|^{-\frac{n-1}{2}+\frac{n-1}{p}}\norm{f(s,\cdot)}_{L^{p'}_{x'}}ds}_{L_t^p}\\
&\les&h^{-(n-1)}\sigma^{-(n-1)+\frac{2n}{p}}\norm{f}_{L^{p'}_{t,x'}}
\end{eqnarray*}
when 
$$p\ne\frac{2(n+1)}{n-1}.$$
At $p=\frac{2(n+1)}{n-1}$, we use Hardy-Littlewood-Sobolev inequality in dimension one to get
\begin{eqnarray*}
\norm{\int W_2(t,s)f(s)ds}_{L^p_{t,x'}}&\les& \norm{\int_{\sigma\les |t-s|\les h} h^{-(n-1)}\sigma^{-\frac{n-1}{2}+\frac{n-1}{p}} |t-s|^{-\frac{n-1}{2}+\frac{n-1}{p}}\norm{f(s,\cdot)}_{L^{p'}_{x'}}ds}_{L_t^p}\\
&\les&h^{-(n-1)}\sigma^{-\frac{n-1}{2}+\frac{n-1}{p}}\norm{f}_{L^{p'}_{t,x'}}.
\end{eqnarray*}

\noindent\textbf{Regime 3: $|t-s|\ge{}Ch$.}
Interpolation between \eqref{eq:3rdL1Linfty} and \eqref{eq:3rdL2L2} gives that
$$\norm{W_{3}(t,s)}_{L^{p'}_{x'}\to L^{p}_{x'}}\lesssim{}h^{-\frac{n-1}{2}}\sigma^{-\frac{n-1}{2}+\frac{n-1}{p}} |t-s|^{-(n-1)+\frac{n-1}{p}}.$$
Hence, from H\"older's inequality, we have that
\begin{align*}
\norm{\int W_3(t,s)f(s)ds}_{L^p_{t,x'}}\les& \norm{\int_{|t-s|\ges h} h^{-\frac{n-1}{2}}\sigma^{-\frac{n-1}{2}+\frac{n-1}{p}} |t-s|^{-(n-1)+\frac{n-1}{p}}\norm{f(s,\cdot)}_{L^{p'}_{x'}}ds}_{L_t^p}\\
\les&h^{-\frac{3(n-1)}{2}+\frac{n+1}{p}}\sigma^{-\frac{n-1}{2}+\frac{n-1}{p}}\norm{f}_{L^{p'}_{t,x'}}
\end{align*}
when $p\ne(n+1)/(n-1)$. But $(n+1)/(n-1)<2$ unless $n=2,3$. So we only need to consider the cases $(n,p)=(2,3)$ and $(n,p)=(3,2)$.

For $(n,p)=(2,3)$, by Hardy-Littlewood-Sobolev inequality in dimension one we get
\begin{align*}
\norm{\int W_3(t,s)f(s)ds}_{L^3_{t,x'}}\les& \norm{\int_{|t-s|\ges h} h^{-\frac{1}{2}}\sigma^{-\frac{1}{6}} |t-s|^{-2/3}\norm{f(s,\cdot)}_{L^{3/2}_{x'}}ds}_{L_t^3}\\
\les& h^{-\frac{1}{2}}\sigma^{-\frac{1}{6}}\norm{f}_{L^{3/2}_{t,x'}}.
\end{align*}
For $(n,p)=(3,2)$, we can not use Hardy-Littlewood-Sobolev inequality as in the above case. However, a similar argument as in Tacy \cite[Proposition 5.1]{T} gives the following estimate with a $\log$ loss.
\begin{align*}
\norm{\int W_3(t,s)f(s)ds}_{L^2_{t,x'}}\les& h^{-1}|\log h|\norm{f}_{L^{2}_{t,x'}}.
\end{align*}
We complete the proof by gathering the above calculations to obtain
$$\norm{\int W_i(t,s)f(s)\,ds}_{L^p_{t,x'}}\le G(h,\sigma)^2\norm{f}_{L^{p'}_{t,x'}},\quad i=1,2,3.$$
\end{proof}

\begin{prop}[Case 1b]\label{dsfs}
Assume the notations in Case 1:
$$\norm{(hD_{x_{1}}-a_{1}(x,hD_{x'}))u(h)}_{L^{2}}\lesssim{}h\norm{u(h)}_{L^{2}},$$
and
$$\norm{(\sigma{}D_{x_{1}}-a_{2}(x,\sigma{}D_{x'}))v(\sigma)}_{L^{2}}\lesssim{}\sigma\norm{v(\sigma)}_{L^{2}},$$
where $\partial^2_{\xi'\xi'}a_1$ and $\partial_{\xi'\xi'}^2a_2$ are positive definite. Suppose that $Ch\leq{}\sigma\leq{}h$, where $C$ is as in Proposition \ref{dsdf}. Then
$$\norm{uv}_{L^{p}}\lesssim{}F_{n,p}(h)\norm{u}_{L^{2}}\norm{v}_{L^{2}},$$
where for $n=2$,
$$F_{2,p}(h)=\begin{cases}
h^{-\frac{1}{2}+\frac{1}{2p}} & \text{for }2\leq{}p\leq{}3,\\
h^{-1+\frac{2}{p}} & \text{for }3\leq{}p\leq{}\infty;
\end{cases}$$
and for $n\geq{}3$, $(n,p)\neq{}(3,2)$,
$$F_{n,p}(h)=h^{-n+1+\frac{n}{p}};$$
and
$$F_{3,2}(h)=h^{-\frac{1}{2}}|\log h|^{\frac12}.$$
\end{prop}

\begin{rmk}
When $\sigma\approx h$, $F_{n,p}(h)\approx G_{n,p}(h,\sigma)$, and in this case Proposition \ref{dsfs} gives the $L^p$ bilinear quasimode estimates in full range.
\end{rmk}

\begin{proof}
Note that since $Ch\leq{}\sigma\leq{}h$ the ratio $\sigma/h\approx1$. Assume that $\|u\|_{L^2}=\|v\|_{L^2}=1$. By scaling $\xi$ we obtain
$$\norm{(hD_{x_{1}} -\tilde{a}_2(x,hD_{x'}))v}_{L^{2}}=O(h),$$
where
$$\tilde{a}_2(x,\xi')=\frac{h}{\sigma}a_2\left(x,\frac{\sigma}{h}\xi\right).$$
Therefore, we can treat both $u$ and $v$ as quasimodes at the same scale. We adopt the notation $t=x_{1}$ and write $x\in\M$ as $(t,x')$. Then consider the function
$$f(t,x',z')=u(t,x')v(t,z').$$
We want to calculate the $L^{p}$ norm of $f$ restricted to the submanifold $x'=z'$. Note that
$$(hD_{t}-a_1(t,x',hD_{x'})-\tilde{a}_2(t,z',hD_{z'}))f(t,x',z')=O_{L^{2}}(h).$$
Therefore, we may directly apply the submanifold restriction estimates of Tacy \cite[Theorem 1.7]{T}. When $n=2$, we are restricting from a $3$-dimensional space to a $2$-dimensional hypersurface, so we obtain
$$F_{2,p}(h)=\begin{cases}
h^{-\frac{1}{2}+\frac{1}{2p}} & \text{for }2\leq{}p\leq{}3,\\
h^{-1+\frac{2}{p}} & \text{for }3\leq{}p\leq{}\infty.\end{cases}$$
For $n\geq{}3$, we are restricting from a $(2n-1)$-dimensional space to an $n$-dimensional submanifold, so we obtain
$$F_{n,p}(h)=h^{-n+1+\frac{n}{p}}.$$
Note that as in Tacy \cite{T} we must concede a $\log$ loss for the case $(n,p)=(3,2)$, (in which case, we restrict a $5$-dimensional space to a $3$-dimensional submanifold).
\end{proof}

\subsection{Propagating in different directions (Case 2)}
Recall the notation $x=(x_{1},x_{2},\bar{x})$ and $\xi=(\xi_1,\xi_2,\bar\xi)$.
\begin{prop}\label{dd}
Assume the notations in Case 2:
$$\norm{(hD_{x_1}-a_{1}(x,hD_{x_{2}},hD_{\bar{x}}))u(h)}_{L^{2}}\lesssim{}h\norm{u(h)}_{L^{2}},$$
and
$$\norm{(\sigma{}D_{x_{2}}-a_{2}(x,\sigma{}D_{x_{1}},\sigma{}D_{\bar{x}}))v(\sigma)}_{L^{2}}\lesssim{}\sigma\norm{v(\sigma)}_{L^{2}},$$
where $\partial^2_{(\xi_2,\bar\xi)}a_1$ and $\partial_{(\xi_1,\bar\xi)}^2a_2$ are non-degenerate. Then
$$\norm{uv}_{L^{p}}\lesssim{}G_{n,p}(h,\sigma)\norm{u}_{L^{2}}\norm{v}_{L^{2}}.$$
\end{prop}

\begin{proof}
Let
$$E_h[u]=(hD_{x_1}-a_{1}(x,hD_{x_{2}},hD_{\bar{x}}))u(h),$$
and
$$E_\sigma[v]=(\sigma{}D_{x_{2}}-a_{2}(x,\sigma{}D_{x_{1}},\sigma{}D_{\bar{x}}))v(\sigma).$$
We set $x_{1}=t_{1}$ and $x_{2}=t_{2}$ and express $u$ and $v$ via the propagators
$$\begin{cases}
(hD_{t_1}-a_{1}(s+t_1,x_{2},\bar{x},hD_{x_{2}},hD_{\bar{x}}))U_{h}(t_{1},s)=0,\\
U_{h}(0,s)=\Id,\end{cases}$$
and
$$\begin{cases}
(\sigma{}D_{t_2}-a_{2}(x_{1},s+t_2,\bar{x},hD_{x_{1}},hD_{\bar{x}}))\widetilde{U}_{\sigma}(t,s)=0,\\
\widetilde{U}_{\sigma}(0,s)=\Id.\end{cases}$$
Again we write $U_{h}(t_{1}):=U_{h}(t_{1},0)$ and $\widetilde{U}_{\sigma}(t_{2}):=\widetilde{U}_{\sigma}(t_{2},0)$ for notational convenience. Now we may write
$$u(t_1,t_2,\bar x)=U_{h}(t_{1})u(0,t_2,\bar{x})+\frac{1}{h}\int_{0}^{t_{1}}U_{h}(t_{1}-s_{1},s_{1})E_{h}[u]\,ds_{1},$$
and
$$v(t_1,t_2,\bar x)=\widetilde{U}_{\sigma}(t_{2})v(t_{1},0,\bar{x})+\frac{1}{\sigma}\int_{0}^{t_{2}}\widetilde{U}_{\sigma}(t_{2}-s_{2},s_{2})E_{\sigma}[v]\,ds_{2}.$$
So
\begin{eqnarray*}
uv(t_1,t_2,x)&=&U_{h}(t_{1})u(0,t_2,\bar{x})\widetilde{U}_{\sigma}(t_{2})v(t_{1},0,\bar{x})+\frac{1}{h}\widetilde{U}_{\sigma}(t_{2})v(t_{1},0,\bar{x})\int_{0}^{t_{1}}U_{h}(t_{1}-s_{1},s_{1})E_{h}[u]\,ds_{1}\\
&&+\frac{1}{\sigma}U_{h}(t_{1})u(0,t_2,\bar{x})\int_{0}^{t_{2}}\widetilde{U}_{\sigma}(t_{2}-s_{2},s_{2})E_{\sigma}[v]\,ds_{2}\\
&&+\frac{1}{\sigma{}h}\int_{0}^{t_{1}}\int_{0}^{t_{2}}U_{h}(t_{1}-s_{1},s_{1})E_{h}[u]\widetilde{U}_{\sigma}(t_{2}-s_{2},s_{2})E_{\sigma}[v]\,ds_{1}ds_{2}.
\end{eqnarray*}
As in Proposition \ref{dsdf}, it is enough to prove 
$$\norm{U_{h}(t_{1})f(t_2,\bar x)\widetilde{U}_{\sigma}(t_{2})g(t_1,\bar x)}_{L^{p}_{t_{1},t_{2},\bar{x}}}\lesssim{}G_{n,p}(h,\sigma)\norm{f}_{L^{2}_{t_{2},\bar{x}}}\norm{g}_{L^{2}_{t_{1},\bar{x}}}.$$
Using the parametrix construction, we have
$$U_{h}(t_1)f(t_2,\bar x)=\frac{1}{h^{n-1}}\int{}e^{\frac{i}{h}(\phi_{1}(t_{1},t_{2},\bar{x},\xi_2,\bar\xi)-s_{2}\xi_{2}-\langle{}\bar{y},\bar{\xi}\rangle)}b_{1}(t_{1},t_{2},\bar{x},\xi_2,\bar\xi)f(s_2,\bar{y})\,ds_{2}d\bar{y}d\xi_2d\bar\xi,$$
and
$$\widetilde{U}_{\sigma}(t_{2})g(t_1,\bar x)=\frac{1}{\sigma^{n-1}}\int{}e^{\frac{i}{\sigma}(\phi_{2}(t_{2},t_{1},\bar{x},\eta_1,\bar\eta)-s_{1}\eta_{1}-\langle{}\bar{w},\bar{\eta}\rangle)}b_{2}(t_{1},t_{2},\bar{x},\eta_1,\bar\eta)g(s_{1},\bar{w})\,ds_{1}d\bar{w}d\eta_1d\bar\eta,$$
where
$$\partial_{t_{1}}\phi_{1}(t_{1},t_{2},\bar{x},\xi_2,\bar\xi)-a_{1}(t_{1},t_{2},\bar{x},\partial_{x_2}\phi_1,\nabla_{\bar x}\phi_{1})=0,\quad{}\phi_{1}(0,t_{2},\bar{x},\xi_2,\bar\xi)=t_{2}\xi_2+\langle{}\bar{x},\bar{\xi}\rangle,$$
and
$$\partial_{t_{2}}\phi_{2}(t_{2},t_{1},\bar{x},\eta_1,\bar\eta)-a_{2}(t_{2},t_{1},\bar{x},\partial_{x_1}\phi_2,\nabla_{\bar x}\phi_{2})=0,\quad{}\phi_{2}(0,t_{1},\bar{x},\eta_1,\bar\eta)=t_{1}\eta_{1}+\langle{}\bar x,\bar\eta\rangle.$$
We will sometimes write $t=(t_{1},t_{2})$ for notational convenience. Therefore, we need to study the $L^{2}(\R^2)\times{}L^{2}(\R^{2(n-2)})\to{}L^{p}(\R^2)\times{}L^{p}(\R^{n-2})$ mapping properties of the operator $W_{h,\sigma}(t)$ given by
\begin{multline}
(W_{h,\sigma}(t)h)(\bar x)\\
=\frac{1}{(h\sigma)^{n-1}}\int{}e^{i\left[\frac{1}{h}(\phi_{1}(t_{1},t_{2},\bar{x},\xi_2,\bar\xi)-s_{2}\xi_{2}-\langle{}\bar{y},\bar{\xi}\rangle)+\frac{1}{\sigma}(\phi(t_{2},t_{1},\bar{x},\eta_1,\bar\eta)-s_{1}\xi_{1}-\langle{}\bar{w},\bar{\eta}\rangle)\right]}b(t_{1},t_{2},\bar{x},\xi_1,\bar\xi,\eta_2,\bar\eta)\times\\
h(s_{1},s_{2},\bar{y},\bar{w})\,d\xi_1d\bar\xi d\eta_2d\bar\eta ds_{1}ds_{2}d\bar{y}d\bar{w}.\label{Whsigmadef2}\end{multline}
To this end, we evaluate $W_{h,\sigma}(t)W_{h,\sigma}^{\star}(s)$:
$$(W_{h,\sigma}(t)W_{h,\sigma}^{\star}(s))f(\bar x)=\int{}W(t,\bar{x},s,\bar{z})f(s,\bar{z})\,d\bar{z}ds.$$
Here,
$$W(t,\bar{x},s,\bar{z})=\frac{1}{(h\sigma)^{2(n-1)}}\int{}e^{i\psi_{h,\sigma}(t,r,s,\bar{x},\bar{y},\bar{w},\bar{z},\xi_2,\bar\xi,\zeta_2,\bar\zeta,\eta_1,\bar\eta,\tau_1,\bar\tau)}b(t,r,s,\bar{x},\bar{z},\xi_2,\bar\xi,\zeta_2,\bar\zeta,\eta_1,\bar\eta,\tau_1,\bar\tau)\,d\Lambda,$$
where
$$d\Lambda=drd\bar{y}d\bar{w}d\xi_2d\bar\xi d\zeta_2d\bar\zeta d\eta_1d\bar\eta d\tau_1d\bar\tau,$$
and
\begin{eqnarray*}
&&\psi_{h,\sigma}(t,r,s,\bar{x},\bar{y},\bar{w},\bar{z},\xi_2,\bar\xi,\zeta_2,\bar\zeta,\eta_1,\bar\eta,\tau_1,\bar\tau)\\
&=&\frac{1}{h}\Big[\phi_{1}(t_{1},t_{2},\bar{x},\xi_2,\bar\xi)-\phi_{1}(s_{1},s_{2},\bar{z},\zeta_2,\bar\zeta)-r_{2}(\xi_{2}-\zeta_{2})-\langle{}\bar{y},\bar{\xi}-\bar{\zeta}\rangle\Big]\\
&&+\frac{1}{\sigma}\Big[\phi_{2}(t_2,t_1,\bar x,\eta_1,\bar\eta)-\phi_{2}(s_2,s_1,\bar z,\tau_1,\bar\tau)-r_{1}(\eta_{1}-\tau_{1})-\langle{}\bar{y},\bar{\eta}-\bar{\tau}\rangle\Big].
\end{eqnarray*}
We use stationary phase to calculate the $(r_{2},\bar{y},\zeta_2,\bar\zeta,r_{1},\bar{w},\tau_1,\bar\tau)$ integral. As the stationary point is always non-degenerate, we obtain
\begin{eqnarray*}
W(t,\bar{x},s,\bar{z})&=&\frac{1}{h^{n-1}}\int{}e^{\frac{i}{h}(\phi_{1}(t_{1},t_{2},\bar{x},\xi_2,\bar\xi)-\phi_{1}(s_{1},s_{2},\bar{z},\xi_{2},\bar{\xi}))}b(t,s,\bar{x},\bar{z},\xi_{2},\bar{x})\,{}d\xi_{2}d\bar{\xi}\\
&&\times\frac{1}{\sigma^{n-1}}\int{}e^{\frac{i}{h}(\phi_{2}(t_{2},t_{1},\bar{x},\eta_{1},\bar{\eta})-\phi_{2}(s_{2},s_{1},\bar{z},\eta_{1}\bar{\eta}))}\tilde{b}(t,s,\bar{x},\bar{z},\eta_1,\bar\eta)\,{}d\eta_1d\bar\eta\\
&:=&K_{h}(t,s,\bar{x},\bar{z})K_{\sigma}(t,s,\bar{x},\bar{z})\end{eqnarray*}
Now as in the proof of Proposition \ref{dsdf}, we analyse $K_{h}$ and $K_{\sigma}$ by studying the critical points of
\begin{equation}\phi_{1}(t_{1},t_{2},\bar{x},\xi_2,\bar\xi)-\phi_{1}(s_{1},s_{2},\bar{z},\xi_{2},\bar{\xi})\quad\mbox{in $(\xi_{2},\bar{\xi})$}\label{2hkernelphasei}\end{equation}
and
\begin{equation}\phi_{2}(t_{2},t_{1},\bar{x},\eta_{1},\bar{\eta})-\phi_{2}(s_{2},s_{1},\bar{z},\eta_{1}\bar{\eta})\quad\mbox{in $(\eta_{1},\bar{\eta})$}\label{2sigmakernelphase}\end{equation}
For both \eqref{2hkernelphasei} and \eqref{2sigmakernelphase}  to have critical points we must have both 
$$0=(t_{2}-s_{2})+(\bar{x}-\bar{y})\left(1+O(|s|)\right)+(t_{1}-s_{1})(\partial_{(\xi_2,\bar\xi)}a_{1}+O(|t-s|)),$$
and
$$0=(t_{1}-s_{1})+(\bar{x}-\bar{y})\left(1+O(|s|)\right)+(t_{2}-s_{2})(\partial_{(\eta_1,\bar\eta)}a_{2}+O(|t-s|)).$$
Since both $|\partial_{(\xi_2,\bar\xi)}a_{1}|$ and $|\partial_{(\eta_1,\bar\eta)}a_{2}|$ are less than $\ve$, to have both critical points we require
$$|t_{2}-s_{2}|\leq{}|t_{1}-s_{1}|\ve,\quad\text{and}\quad|t_{1}-s_{1}|\leq{}|t_{2}-s_{2}|\ve.$$
Clearly, if $\ve$ is chosen small enough this is impossible. Therefore, we are always able to integrate by parts in one of $(\xi_2,\bar\xi)$ and $(\eta_1,\bar\eta)$. We split the kernel into two parts
$$W(t,\bar{x},s,\bar{z})=W_{1}(t,\bar{x},s,\bar{z})+W_{2}(t,\bar{x},s,\bar{z}),$$
where
$$W_{1}(t,\bar{x},s,\bar{z})=\chi\left(\frac{|t_{2}-s_{2}|}{|t_{1}-s_{1}|}\right)W(t,\bar{x},s,\bar{z}),$$
and
$$W_{2}(t,\bar{x},s,\bar{z})=\left(1-\chi\left(\frac{|t_{2}-s_{2}|}{|t_{1}-s_{1}|}\right)\right)W(t,\bar{x},s,\bar{z}).$$
On the support of $W_{1}(t,\bar{x},s,\bar{z})$, we have $|t_{2}-s_{2}|\leq{}|t_{1}-s_{1}|$ and therefore we cannot find a critical point in $(\eta_1,\bar\eta)$. So integrating by parts we obtain
$$|K_{\sigma}(t,s,\bar{x},\bar{z})|\lesssim{}\sigma^{-(n-1)}\left(1+\frac{|t_{2}-s_{2}|}{\sigma}\right)^{-N}\left(1+\frac{|\bar{x}-\bar{z}|}{\sigma}\right)^{-N}$$
and so using the trivial $h^{-(n-1)}$ estimate on $K_{h}$ we have
$$|W_{1}(t,\bar{x},s,\bar{z})|\lesssim{}\sigma^{-(n-1)}h^{-(n-1)}\left(1+\frac{|t_{2}-s_{2}|}{\sigma}\right)^{-N}\left(1+\frac{|\bar{x}-\bar{z}|}{\sigma}\right)^{-N}.$$
Using this estimate we obtain the $L^1_{\bar x}\to L^\infty_{\bar x}$ and $L^2_{\bar x}\to L^2_{\bar x}$ norms of $W_1(t,s)$ via Young's inequality.  Then we may use the Strichartz formalism as in Case 1 to get the $L^{p'}_{t,\bar x}\to{}L^{p}_{t,\bar x}$ norm of $W_{1}$:
$$\norm{W_{1}}_{L^{p'}_{t,\bar x}\to{}L^{p}_{t,\bar x}}\lesssim{}h^{-(n-1)}\sigma^{-(n-1)+\frac{2n}{p}}.$$
Now on the support of $W_{2}$ we have $|t_{1}-s_{1}|<|t_{2}-s_{2}|$ therefore there may only be critical points for $K_{\sigma}$ but not $K_{h}$.  This is the same as in the second regime in the proof of Proposition \ref{dsdf} so we inherit those estimates here.
\end{proof}

\section{Sharpness of the $L^p$ bilinear estimates}\label{sec:sharpness}

\subsection{Flat Model}
We study the flat model, that is, localised quasimodes of the Laplacian in $\R^n$, to gain insight into sharp examples. 

Suppose that $u$ is an $L^{2}$ normalised $O_{L^2}(h)$ quasimode of $\Delta_{\R^n}$. Then under Fourier transform
$$\norm{(|\xi|^{2}-1)\mathcal{F}_{h}(u)}\lesssim{}h,$$
where $\mathcal{F}_{h}$ is the semiclassical Fourier transform defined as
$$\mathcal{F}_{h}(u)(\xi)=\frac{1}{(2\pi{}h)^{n/2}}\int_{\R^n}e^{-\frac{i}{h}\langle{}x,\xi\rangle}u(x)\,dx.$$ 
Then $\mathcal{F}_{h}[u]$ must be located near the sphere of radius $1$ in the $\xi$-space. We create a family of quaismodes indexed by $\alpha$ which controls the degree of angular dispersion of $\xi$. Write $\xi=(r,\omega)$ where $\omega\in{}S^{n-1}$ and set the coordinate system so that $\omega_{0}$ corresponds with the unit vector in the $\xi_{1}$ direction. Let
$$\chi_{\alpha}^{h}(r,\omega)=\begin{cases}
1 & \text{if }|r-1|<h,|\omega-\omega_{0}|<h^{\alpha},\\
0 & \mbox{otherwise}.\end{cases}$$
Then set
$$f^h_{\alpha}(\xi)=f^{h}_{\alpha}(r,\omega)=h^{-1/2-\alpha(n-1)/2}\chi_\alpha^h(r,\omega).$$
Note that $f^{h}_{\alpha}$ is $L^{2}$ normalised. Now set
$$T^h_{\alpha}(x)=\mathcal{F}_{h}^{-1}(f^h_{\alpha})(x)=\frac{1}{(2\pi{}h)^{n/2}}\int_{\R^n}e^{\frac{i}{h}\langle{}x,\xi\rangle}f_{\alpha}^h(\xi)\,d\xi.$$
$T^h_{\alpha}$ is an $L^{2}$ normalised $O(h)$ quasimode of $\Delta_{\R^n}$. We may write
$$T^h_{\alpha}(x)=\frac{h^{-1/2-\alpha(n-1)/2-n/2}e^{\frac{i}{h}x_{1}}}{(2\pi)^{n/2}}\int_{\R^n}e^{\frac{i}{h}(x_{1}(\xi_{1}-1)+\langle{}x,\xi'\rangle)}\chi_{\alpha}^h(\xi)\,d\xi.$$
Note that if $|x_{1}|<\ve{}h^{1-2\alpha}$ and $|x'|<\ve{}h^{1-\alpha}$ for sufficiently small $\ve>0$, the factor
$$e^{\frac{i}{h}(x_{1}(\xi_{1}-1)+\langle{}x',\xi'\rangle)}$$
does not oscillate so in this region
$$|T^h_{\alpha}(x)|>ch^{-(n-1)/2+\alpha(n-1)/2}.$$
\begin{figure}[h!]\label{Talphafig}
\includegraphics[scale=0.4]{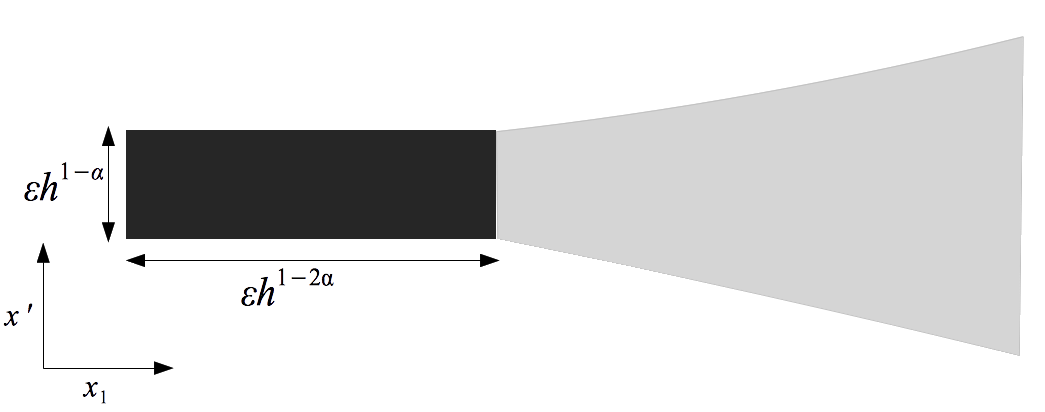}
\caption{$T^{h}_{\alpha}$ is localised so that is large in a $h^{1-2\alpha}\times{}(h^{1-\alpha})^{n-1}$ tube}
\end{figure}

Now suppose we have two semiclassical parameters $\sigma<h$. For each we have a family of tubes $T^\sigma_{\alpha}$ and $T_{\beta}^h$.  We will construct sharp examples by studying the products of such tubes.

\subsubsection{$L^p$ bilinear quasimode estimates for large $p$}
First we treat the case when $p>\frac{2(n+1)}{n-1}$. We set $\alpha=\beta=0$ and consider the product $T_{0}^\sigma T_{0}^h$. 
\begin{figure}[H]\label{highpfig}
\includegraphics[scale=0.5]{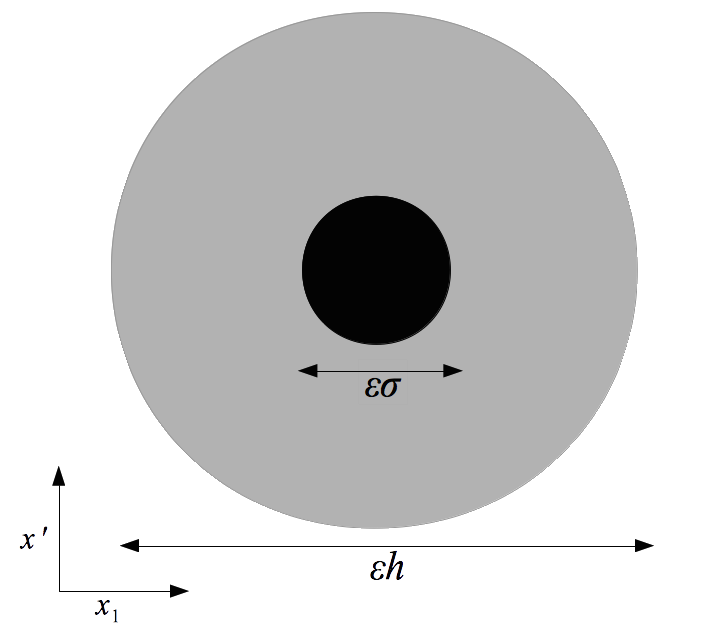}
\caption{To saturate the high $p$ estimates we superimpose $T^{h}_{0}$ and $T^{\sigma}_{0}$}
\end{figure}
For $|x|<\ve\sigma$, we have
$$|T_{0}^\sigma(x)|\gtrsim\sigma^{-(n-1)/2},\quad\text{and}\quad|T_{0}^h(x)|\gtrsim h^{-(n-1)/2}.$$
So
$$\|T_{0}^\sigma T_{0}^h\|_{L^{p}}\gtrsim\sigma^{-(n-1)/2}h^{-(n-1)/2}\sigma^{n/p}=c_{3}h^{-(n-1)/2}\sigma^{-(n-1)/2-n/p},$$
which therefore saturates the estimates of Theorem \ref{thm:bilinear} when $p\geq{}\frac{2(n+1)}{n-1}$.

\subsubsection{$L^p$ bilinear quasimode estimates for small $p$ in dimension $n\ge3$ and midrange $p$ in dimension two}\label{sec:qmmidrange}
We now construct saturating examples for $2\leq{}p\leq{}\frac{2(n+1)}{n-1}$ in dimension $n\geq{}3$, these examples also saturate the estimates for $3\leq{}p\leq{}6$ in dimension two. We choose $T^h_0$. Then we select the tube $T_{\alpha_{h}}^\sigma$ such that 
$$\sigma^{1-2\alpha_{h}}=h.$$ 
\begin{figure}[H]\label{midpfig}
\includegraphics[scale=0.3]{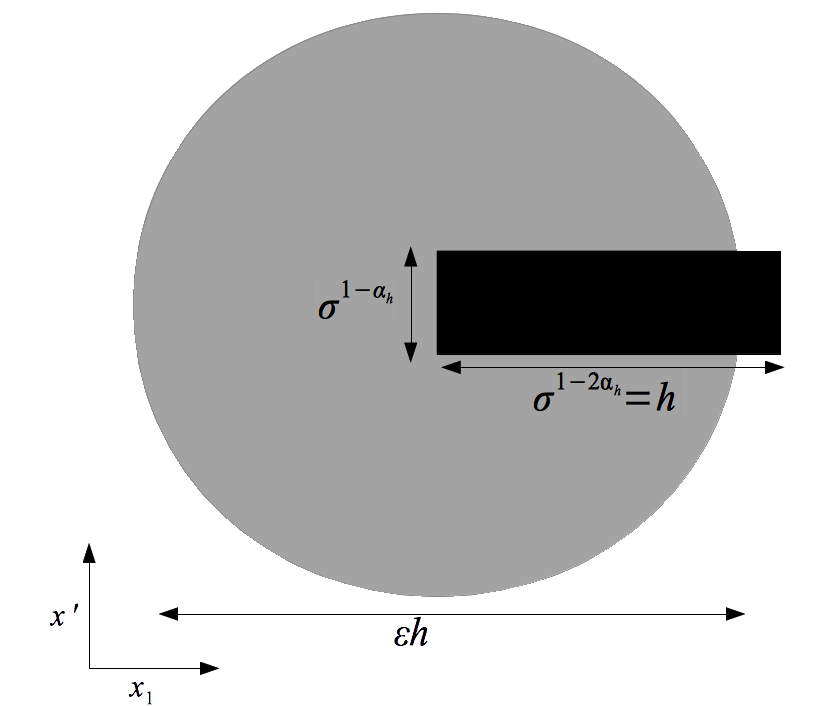}
\caption{To saturate the mid-range $p$ estimates we superimpose $T^{h}_{0}$ and $T^{\sigma}_{\alpha_{h}}$ where $\alpha_{h}$  is choose such that $\sigma^{1-2\alpha_{h}}=h$}
\end{figure}

Then for $|x_{1}|<\ve{}h$ and $|x'|<\ve{}\sigma^{1-\alpha}$, we have 
$$|T_{0}^h(x)|\gtrsim h^{-(n-1)/2},\quad\text{and}\quad|T_{\alpha_{h}}^\sigma(x)|\gtrsim\sigma^{-(n-1)/2+\alpha(n-1)/2}.$$
So
\begin{eqnarray*}
\norm{T_{0}^hT_{\alpha_h}^\sigma}_{L^{p}}&\gtrsim&h^{-(n-1)/2}\sigma^{-(n-1)/2+\alpha(n-1)/2}{}h^{1/p}\sigma^{(1-\alpha)(n-1)/p}\\
&=&Ch^{-3(n-1)/4+(n+1)/2p}h^{(n-1)/4-(n-1)/2p}\sigma^{-(n-1)/2+(n-1)/p+\alpha((n-1)/2-(n-1)/p)}\\
&=&Ch^{-3(n-1)/4+(n+1)/2p}\sigma^{-(n-1)/4+(n-1)/2p}.
\end{eqnarray*}
This saturates the sharp $L^p$ bilinear quasimode estimates for $2\leq{}p\leq{}\frac{2(n+1)}{n-1}$ in dimension $n\geq{}3$ and $3\le p\le6$ in dimension two.

\subsubsection{$L^p$ bilinear quasimode estimates for $2\le p\le3$ in dimension two}
Here we set $\alpha=\beta=1/2$ and consider the product $T_{1/2}(\sigma)T_{1/2}(h)$. 
\begin{figure}[H]\label{lowpfig}
\includegraphics[scale=0.4]{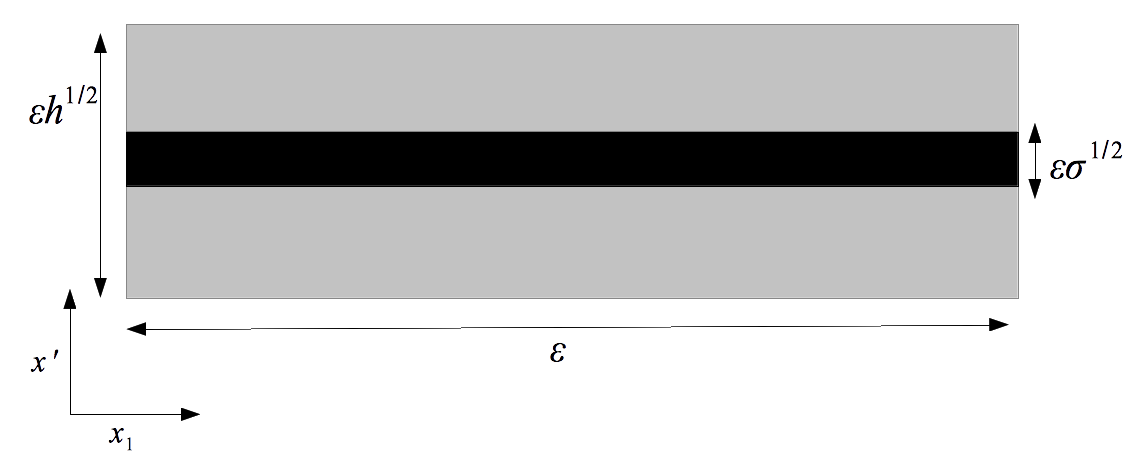}
\caption{To saturate the low $p$ estimates we superimpose $T^{h}_{1/2}$ and $T^{\sigma}_{1/2}$}
\end{figure}
For $|x_{1}|<\ve$ and $|x_{2}|<\ve{}\sigma^{1/2}$, we have
$$|T_{1/2}^\sigma(x)|\gtrsim\sigma^{-1/4},\quad\text{and}\quad|T_{1/2}(h)|\gtrsim h^{-1/4}.$$
So
$$\norm{T_{1/2}^\sigma T_{1/2}^h}_{L^{p}}\gtrsim h^{-1/4}\sigma^{-1/4}\sigma^{1/2p}=Ch^{-1/4}\sigma^{-1/4+1/2p},$$
which saturates the estimates for $2\leq{}p\leq{}3$ in dimension two.

\subsection{Spherical harmonics}\label{sec:SH}
In this subsection, we construct eigenfunctions on the sphere $\s^2$ that saturate the $L^p$ bilinear eigenfunction estimates in Theorem \ref{thm:efnbilinear}. These examples also reflect the same behaviour as the flat model in the previous subsection.

We first recall some standard facts about spherical harmonics as the eigenfunctions on the sphere.  The spherical harmonics are the homogeneous harmonic polynomials restricted on the sphere; we use $\SH_k$ to denote the set of such functions with homogeneous degree $k$. For each $u\in\SH_k$, $u$ is an eigenfunction of $-\Delta_{\s^2}$ on $\s^2$ with eigenvalue $k(k+1)$:
$$-\Delta_{\s^2} u=k(k+1)u.$$
Therefore, the eigenfrequency $\lambda=\sqrt{k(k+1)}\approx k$. The multiplicity of the eigenvalue $k(k+1)$ is the dimension of its eigenspace, $\dim\SH_k=2k+1$. We use  spherical coordinates $\phi\in[0,\pi]$ and $\theta\in[0,2\pi)$ so that $\s^2\ni x=(\sin\phi\cos\theta,\sin\phi\sin\theta,\cos\phi)$. One can write the standard orthonormal basis of $\SH_k$ as $\{Y_m^k\}_{m=-k}^k$:
$$Y_m^k(\phi,\theta)=C_{k,m}P_k^m(\cos\phi)e^{im\theta},$$
in which $C_{k,m}$ is the $L^2$ normalisation factor, and $P_k^m$ is the associated Legendre polynomial. We remark that the phase of $Y_m^k$ is constant on each longitude (i.e. $\theta=\const$), and their modulus is constant on each latitude (i.e. $\phi=\const$). Two special cases of spherical harmonics are as follows.
\begin{enumerate}[(i).]
\item $m=0$: $Z_k=Y_0^k$ are called zonal harmonics. $Z_k$ concentrates on the two antipodal points $\phi=0$ and $\phi=\pi$. They achieve the maximal norm growth in Sogge's $L^p$ estimate \eqref{eq:Sogge} for large $p$: 
$$\|Z_k\|_{L^p}\approx k^{2(1/2-1/p)-1/2},\quad6\le p\le\infty.$$
If $|\phi|\lesssim k^{-1}$, i.e. within a $k^{-1}$ neighborhood of the north pole, $|Z_k|\gtrsim k^{1/2}$. Similar estimates holds also around the south pole.
\item $m=\pm k$: $Q_{\pm k}=Y^k_{\pm k}$ are called highest weight spherical harmonics or Gaussian beams. $Q_{\pm k}$ concentrate in a $k^{-1/2}$ neighborhood of the equator $\phi=\pi/2$ and achieve the maximal norm growth in Sogge's $L^p$ estimate \eqref{eq:Sogge} for small $p$: 
$$\|Q_{\pm k}\|_{L^p}\approx k^{(1/2-1/p)/2},\quad2<p\le6.$$
Notice that $Q_{\pm k}$ decreases exponentially away from the concentration and $\|Q_{\pm k}\|_{L^\infty}\approx k^{1/4}$.
\end{enumerate}

To construct the sharp examples for $L^p$ bilinear eigenfunction estimates in Theorem \ref{thm:efnbilinear}, we divide the range of $p$ into large $p$ ($p\ge6$), midrange $p$ ($3\le p\le6$), and small $p$ ($2\le p\le3$). 

\subsubsection{$L^p$ bilinear eigenfunction estimates for $p\ge6$}
Let $\lfloor\lambda\rfloor$ be the largest integer that is smaller than $\lambda$. Write $u=Z_{\lfloor\lambda\rfloor}$ and $v=Z_{\lfloor\mu\rfloor}$. Then
$$|u(\phi,\theta)|\gtrsim\lambda^\frac12,\quad\text{if}\quad|\phi|<\ve\lambda^{-1},$$
and
$$|v(\phi,\theta)|\gtrsim\mu^\frac12,\quad\text{if}\quad|\phi|<\ve\mu^{-1}.$$
So 
$$\|uv\|_{L^p}\gtrsim\lambda^\frac12\mu^{\frac12-\frac2p},$$
which saturates the estimate in Theorem \ref{thm:efnbilinear} when $p\ge6$.

\subsubsection{$L^p$ bilinear eigenfunction estimates for $2\le p\le3$}
Let $u=Q_{\lfloor\lambda\rfloor}$ and $v=Q_{\lfloor\mu\rfloor}$. Then
$$|u(\phi,\theta)|\gtrsim\lambda^\frac14,\quad\text{if}\quad\left|\phi-\frac\pi2\right|<\ve\lambda^{-\frac12},$$
and
$$|v(\phi,\theta)|\gtrsim\mu^\frac14,\quad\text{if}\quad\left|\phi-\frac\pi2\right|<\ve\mu^{-\frac12}.$$
So 
$$\|uv\|_{L^p}\gtrsim\lambda^\frac14\mu^{\frac14-\frac{1}{2p}},$$
which saturates the estimate in Theorem \ref{thm:efnbilinear} when $2\le p\le3$. We remark that on $\s^n$ when $n\ge3$, Burq-G\'erard-Tzvetkov \cite{BGT5} proved that two zonal harmonics saturate the $L^2$ bilinear eigenfunction estimates in Theorem \ref{thm:BGT} (modulo $\log$ when $n=3$).

\subsubsection{$L^p$ bilinear eigenfunction estimates for $3\le p\le6$}
For the smaller eigenfrequency $\lambda$, we let $u=Z_{\lfloor\lambda\rfloor}$. Then 
$$|u(\phi,\theta)|\gtrsim\lambda^\frac12,\quad\text{if}\quad|\phi|<\ve\lambda^{-1}.$$
Let $\alpha\in[0,1/2)$. We set
\begin{equation}\label{eq:lambdamu}
\mu=\lambda^{\frac{1}{1-2\alpha}},
\end{equation} 
and construct the eigenfunction $v$ such that
\begin{equation}\label{eq:midrangev}
|v(\phi,\theta)|\gtrsim\mu^{\frac{1-\alpha}{2}},\quad\text{if}\quad|x_1|<\ve\lambda^{-1}=\ve\mu^{-(1-2\alpha)}\ \text{and}\ |x_2|<\ve\mu^{-(1-\alpha)}.
\end{equation}
Recall that in our notation, $\s^2\ni(x_1,x_2,x_3)=(\sin\phi\cos\theta,\sin\phi\sin\theta,\cos\phi)$. One sees immediately that this is the eigenfunction version of the flat modal in \S\ref{sec:qmmidrange}. In  view of \eqref{eq:lambdamu}
$$\|uv\|_{L^p}\gtrsim\lambda^\frac12\mu^{\frac{1-\alpha}{2}}\lambda^{-\frac1p}\mu^{-\frac{1-\alpha}{p}}\ge\lambda^{\frac34-\frac{3}{2p}}\mu^{\frac14-\frac{1}{2p}}.$$
 To construct the required spherical harmonic $v$, we use linear combination of Gaussian beams. 

Set $k=\lfloor\mu\rfloor$. According to the different propagating directions of $Q_k$ and $Q_{-k}$, we say that $Q_k$ has the north pole $\phi=0$ as its pole and $Q_{-k}$ has the south pole $\phi=\pi$ as its pole. Since $\Delta$ is rotational invariant, given any point $p\in\s^2$, we can find a Gaussian beam with pole $p$ by rotating $Q_k$. Two Gaussian beams with the same pole only differ by a phase shift. (See Han \cite[Lemma 8]{Ha} for more details.) Hence, a Gaussian beam in $\SH_k$ is uniquely determined by its pole and the phase of some point on the sphere.

Next we need a theorem in Han \cite{Ha}, which shows that for a family of well-separated poles $\{p_j\}_{j=1}^m\subset\s^2$, the corresponding Gaussian beams $\{q_j\}_{j=1}^m\subset\SH_k$ are almost orthogonal. This enables us to estimate the $L^2$ norm of a superposition of well-separated Gaussian beams.

\begin{lemma}\label{lemma:almost}
There exists a positive number $d>0$ such that the following statement is true. For any family of well-separated poles $\{p_j\}_{j=1}^m\subset\s^2$ satisfying 
$$\dist(p_i,p_j)\ge d\cdot k^{-\frac12}\quad\text{for all }i,j,$$ the eigenvalues of the Hermitian matrix $(\langle q_i,q_j\rangle)$ are all in $[1/2,2]$. Here, $q_j$ is the Gaussian beam with pole $p_j$ and $\langle\cdot,\cdot\rangle$ denotes the inner product in $L^2(\s^2)$.
\end{lemma}

The proof can be found in Han \cite[Section 3]{Ha}. 

Let $m=\lfloor k^{1/2-\alpha}\rfloor$. We choose the poles $\{p_j\}_{j=1}^m$ around $\theta=\pi/2$ on the equator $\phi=\pi/2$ with distance between each pair satisfying $\dist(q_i,q_j)=d\cdot k^{-1/2}$. Then all $p_j$ fall into a $\theta_0$-neighborhood of $\theta=\pi/2$ on the equator, where 
$$\theta_0=m\cdot d\cdot k^{-\frac12}\approx k^{-\alpha}.$$ 
Without loss of generality, assume that $p_1=(0,1,0)$, that is, $\phi=\pi/2$ and $\theta=\pi/2$ in spherical coordinates. Then the corresponding Gaussian beam $q_1$ concentrates around the great circle defined by the equation $\theta=0,\pi$. Any other corresponding Gaussian beam $q_j$ concentrates on the great circle which has an angle $\le\theta_0$ with the one of $q_1$. And $q_j's$ intersect at the north pole $p_n$. Write
$$w=\frac{1}{\sqrt m}\sum_{j=1}^mq_j.$$
Then thanks to Lemma \ref{lemma:almost}, we have
\begin{equation}\label{eq:wL2}
\|w\|_{L^2}^2=\frac1m\left\|\sum_{j=1}^mq_j\right\|_{L^2}^2=\frac1m\left\langle\sum_{i=1}^mq_i,\sum_{j=1}^m\overline q_j\right\rangle\le2.
\end{equation}
We now set the phase of all $q_j's$ at the north pole to be $e^{ik0}=1$. Hence,
$$w(p_n)=\frac{1}{\sqrt m}\sum_{j=1}^mq_j(p_n)\gtrsim k^{\frac{1-\alpha}{2}}.$$
We need to show that the above lower bound holds in the neighbourhood of $p_n$:
$$S=\{x=(x_1,x_2,x_3)\in\s^2:|x_1|<\ve k^{-(1-2\alpha)}\ \text{and}\ |x_2|<\ve k^{-(1-\alpha)}\}$$
for sufficiently small $\ve$. Notice that $S$ falls into the concentration tube of every $q_j$ since $\alpha\in[0,1/2)$ and $|x_2|<\ve k^{-(1-\alpha)}<k^{-1/2}$. Therefore, fix $x\in S$, then
\begin{equation}\label{eq:qjx}
|q_j(x)|\gtrsim k^\frac14\quad\text{for all }j.
\end{equation}
To determine the phases of $q_j(x)$ for $j=1,...,m$, we let $p_j$ be the new north pole and denote $d_j(x)$ as the longitudinal difference of $x$ and $p_n$ in this new coordinate system. Hence, the phase difference of $q_j(x)$ and $q_j(p_n)$ is $e^{ikd_j(x)}$, therefore the phase of $q_j(x)$ is $e^{ikd_j(x)}$ since the phase of $q_j(p_n)$ is set to be $1$.

We observe that for any $x\in S$,
$$|d_1(x)-d_j(x)|\lesssim\ve k^{-(1-\alpha)}\cdot\theta_0\le\ve k^{-1}.$$
Thus, by choosing $\ve$ small, the real part
$$\Re\left(e^{ik(d_j(x)-d_1(x))}\right)\ge\frac12\quad\text{for all }j=1,...,m.$$
Hence,
$$w(x)=\frac{1}{\sqrt m}\sum_{j=1}^m|q_j(x))|e^{ikd_j(x)}=\frac{1}{\sqrt m}e^{ikd_1(x)}\sum_{j=1}^m|q_j(x)|e^{ik(d_j(x)-d_1(x))},$$
and then in view of \eqref{eq:qjx}
$$\Re\left(e^{-ikd_1(x)}w(x)\right)=\frac{1}{\sqrt m}\sum_{j=1}^m|q_j(x)|\Re\left(e^{ik(d_j(x)-d_1(x))}\right)\gtrsim k^{\frac{1-\alpha}{2}}.$$
 Let $v=w/\|w\|_{L^2}$. Then by \eqref{eq:wL2}, 
$$|v(x)|\gtrsim k^{\frac{1-\alpha}{2}}\quad\text{for all }x\in S,$$
which gives the required spherical harmonic in \eqref{eq:midrangev}.

\section{Appendix: Semiclassical analysis}\label{sec:SA}
In this appendix, we provide the background on semiclassical analysis that is used in this paper. We refer to Zworski \cite{Zw} for a complete treatment in this subject.

Semiclassical analysis provides a valuable framework in which to consider high frequency problems. The key idea is to scale a fixed frequency $\lambda\to{}1$. This scaling gives rise to a semiclassical Fourier transform (with $\lambda=1/h$)
$$\mathcal{F}_{h}[u](\xi)=\frac{1}{(2\pi{}h)^{n/2}}\int{}e^{-\frac{i}{h}\langle{}x,\xi\rangle}u(x)\,dx,$$
and its inverse
$$\mathcal{F}_{h}^{-1}[f](x)=\frac{1}{(2\pi{}h)^{n/2}}\int{}e^{\frac{i}{h}\langle{}x,\xi\rangle}f(\xi)\,d\xi.$$
Note that as with the standard Fourier transform there is some variance in the literature as to the pre-factor. We choose $(2\pi{}h)^{-n/2}$ as this is exactly the right factor to preserve the $L^{2}$ norms. That is,
$$\norm{\mathcal{F}_{h}u}_{L^{2}}=\norm{u}_{L^{2}}.$$
Under this scaling the standard relationships between differentiation and multiplication are retained (although they now also include a scale factor). That is,
$$\mathcal{F}_{h}[hD_{x_{i}}u]=\xi_{i}\mathcal{F}_{h}[u]\quad\text{and}\quad\mathcal{F}_{h}^{-1}[x_{i}f]=-hD_{\xi_{i}}\mathcal{F}^{-1}_{h}[f].$$
Therefore constant coefficient differential operators can be written as
$$a(hD)u(x)=\frac{1}{(2\pi{}h)^{n}}\int{}e^{\frac{i}{h}\langle{}x-y,\xi\rangle}a(\xi)u(y)\,d\xi{}dy.$$
Naturally we want to extend to semiclassical pseudodifferential operators $a(x,hD)$:
$$a(x,hD)u(x)=\frac{1}{(2\pi{}h)^{n}}\int{}e^{\frac{i}{h}\langle{}x-y,\xi\rangle}a(x,\xi)u(y)\,d\xi{}dy.$$
Therefore we must define some symbol classes for the symbols $a(x,\xi)$ to live in.

Like the standard pseudodifferential calculus our symbols will be classical observables on phase space. Suppose that $(\M,g)$ is a Riemannian manifold. An element in the cotangent bundle $T^*\M$ is denoted as $(x,\xi)$ with $x\in\M$ and $\xi\in T^*_x\M$. We write $|\xi|_x$ as the induced norm of $\xi\in T_x^*\M$ by the Riemannian structure $g$. When we work locally, as we often do, we may associated $T^*\M$ with patches of $\R^{2n}$. In this paper we concern ourselves with functions that are semiclassically localised.
\begin{defn}[Semiclassical localisation]
Let $K$ be a compact subset of $T^*\M$. We say that a tempered family $\{u(h)\}$ is localized to $K$ in the phase space if there exists a function $\chi\in C^\infty_0(K)$ for which
$$u(h)=\chi(x,hD)u(h)+O_{L^2}(h^\infty).$$
\end{defn}

This means that we may assume that any symbols we use have compact support. In the standard pseudodifferential calculus, symbols with compact support are smoothing and therefore do not give rise to interesting mathematics. However in the semiclassical calculus the $h$ scaling means that even smooth and compactly supported symbols can indeed give rise to very interesting mathematics. 

We need to define a sensible class of symbols and a procedure to quantise them. That is, to associate them with a semiclassical pseudodifferential operator.

\subsection{Symbol classes}\label{sec:symbol}
 In the analogy to the standard calculus one may define symbols $a(x,\xi)$ that on compact subsets $K$ of $T^*\M$ 
$$\sup_{x\in K,\xi\in T_x^*\M}|\partial^\alpha_x\partial^\beta_\xi a(x,\xi)|\le C_{\alpha,\beta,K}\,\langle\xi\rangle^{m-|\beta|}$$
for some $C_{\alpha,\beta,K}$ independent of $h$. However these classes don't capture the behaviour of the symbol with respect to the semiclassical parameter. To motivate our choice of symbol class, consider the standard symbols
$$a(x,\xi;\lambda)=\left(1+|\xi|^{2}\right)^\frac m2\rho\left(\frac{|\xi|}{\lambda}\right).$$
where $\rho\in C^\infty_0(\R)$ has support near $1$. Hence, these symbols are localised near $|\xi|=\lambda$. We have that
$$a(x,D)u(x)=\frac{1}{(2\pi)^{n}}\int{}e^{i\langle{}x-y,\xi\rangle}\left(1+|\xi|^{2}\right)^\frac m2\rho\left(\frac{|\xi|}{\lambda}\right)u(y)\,d\xi{}dy,$$
re-scaling $\xi\to\lambda\xi$ and setting $\lambda^{-1}=h$ we have
$$a(x,D)u(x)=\frac{1}{(2\pi{}h)^{n}}\int{}e^{\frac{i}{h}\langle{}x-y,\xi\rangle}\left(1+\frac{|\xi|^{2}}{h^{2}}\right)^\frac m2\rho(|\xi|)u(y)\,d\xi dy=a_{h}(x,hD)u(x),$$
where
$$a_{h}(x,\xi;h)=\left(1+\frac{|\xi|^{2}}{h^{2}}\right)^\frac m2\rho(|\xi|).$$
Note that the symbol $a_{h}(x,\xi)$ is localised near $|\xi|=1$ so
$$|a_{h}(x,\xi;h)|\les h^{-m}.$$
Therefore the level of smoothness (captured by the order $m$) of a symbol in the standard calculus translates to decay in powers of $h$ when we consider it semiclassically. Therefore we define a set of symbol spaces of compactly supported symbols to reflect this relationship.

\begin{defn}
Let $m\in\R$. We define the symbol classes
$$S^{m}(\M)=\{a\in{}C_{0}^{\infty}(T^*\M)\mid|\partial^{\alpha}a|\leq{}C_{\alpha}h^{-m}\}.$$
\end{defn}

\begin{enumerate}
\item If $m=0$, we denote $S^m(\M)$ by $S(\M)$.
\item We denote $S^{-\infty}(\M)=\cap_{m\in\R}S^m(\M)$ and $S^\infty(\M)=\cup_{m\in\R}S^m(\M)$. 
\end{enumerate}
If $a(x,\xi;h)$ is independent of $h$ we write it as $a(x,\xi)$. We refer the reader the Zworski \cite{Zw} for definitions of symbols that are not compactly supported or that lack of smoothness as $h\to{}0$. 

\subsection{Semiclassical pseudodifferential operators}\label{sec:SDOs}
Having an appropriate class of symbols we can now associate every symbol with a semiclassical pseudodifferential operator (i.e. \SDO). As in the standard calculus there is some choice to the quantisation procedure.  

\begin{defn}[Standard and Weyl quantisations]
Given $a\in S^m(\R^n)$, we define
\begin{enumerate}[(i).]
\item the left (or standard) quantisation as
$$a(x,hD)u(x)=\frac{1}{(2\pi h)^n}\int_{\R^2n}e^{\frac{i}{h}\langle{}x-y,\xi\rangle}a\big(x,\xi;h\big)u(y)\,d\xi dy\quad\text{for }u\in\mathcal S(\R^n);$$
\item the Weyl quantisation as
$$a^w(x,hD)u(x)=\frac{1}{(2\pi h)^n}\int_{\R^2n}e^{\frac{i}{h}\langle{}x-y,\xi\rangle}a\left(\frac{x+y}{2},\xi;h\right)u(y)\,d\xi dy\quad\text{for }u\in\mathcal S(\R^n).$$
If the symbol is instead in $S^{m}(\M)$ for $(\M,g)$ a smooth Riemann manifold we may keep this definition replacing $\R^{2n}$ with $T^*\M$. 
\end{enumerate}
\end{defn}

While there are other choices of quantisation these two choices are often the most useful. The Weyl quantisaton has the nice property that $a^w(x,hD)$ is self-adjoint if $a$ is a real-valued symbol. The standard quantisation has the nice property that
$$a(x,hD)u=\mathcal{F}_{h}^{-1}\left(a(x,\xi;h)\mathcal{F}_{h}[u]\right).$$

\begin{rmk}
If $a(x,\xi;h)\in{}S^{m}(\M)$ then one can show (by almost orthogonality) that 
$$\|a(x,hD)||_{L^{2}\to{}L^{2}}\les h^{-m}.$$
In particular, if $a(x,\xi;h)\in{}S(\M)$ then the $L^{2}\to{}L^{2}$ mapping norm of $a(x,hD)$ is bounded independent of $h$.
\end{rmk}

We now define the set $\Psi^m(\M)$ of semiclassical pseudodifferential operators with symbols in $S^m(\M)$, and establish the correspondence of $A\in\Psi^m(\M)$ and its semiclassical principal symbol $a$. The correspondence is one-to-one modulo lower order terms. Denote
$$a=\sigma(A):\Psi^m(\M)\to S^m(\M)/S^{m-1}(\M),$$
and its right inverse, a non-canonical quantisation map for $a\in S^m(\M)$:
$$A=\Op_h(a):S^m(\M)\to\Psi^m(\M).$$
$\sigma(A)$ is called the principal symbol of $A$. It is modulo $S^{m-1}(\M)$ unique under change of quantisations and change of local coordinates. In the same fashion in \S\ref{sec:symbol},
\begin{enumerate}
\item if $m=0$, we denote $\Psi^m(\M)$ by $\Psi(\M)$;
\item we denote $\Psi^{-\infty}(\M)=\cap_{m\in\R}\Psi^m(\M)$ and $\Psi^\infty(\M)=\cup_{m\in\R}\Psi^m(\M)$. In this context the elements in $\Psi^{-\infty}(\M)$ are referred as smoothing operators;
\end{enumerate}

We sometimes abuse notation somewhat and neglect to show the cutoff functions that provide the compact support. For example we may work with the Laplacian and refer to its symbol as $|\xi|_{x}^{2}$ when in fact in our localised setting we are really working with the operator $p(x,hD)$ with symbol
$$p(x,\xi)=\rho(|\xi|_{x})|\xi|_{x}^{2}$$
where $\rho\in C^\infty_0(\R)$. 

The usual operations involving semiclassical pseudodifferential operators are as follows. Let $A\in\Psi^m(\M)$ and $B\in\Psi^{m'}(\M)$.
\begin{enumerate}
\item Let $A^\star$ be the adjoint operator of $A$ in $L^2(\M)$. Then
\begin{equation}\label{eq:SDOadjoint}
\sigma(A^\star)=\ol{\sigma(A)}+r(x,\xi;h),
\end{equation}
where the error term $r(x,\xi;h)\in{}S^{m-1}(\M)$.
\item
\begin{equation}\label{eq:SDOproduct}
\sigma(AB)=\sigma(A)\sigma(B)+r(x,\xi;h),
\end{equation}
where the error term $r(x,\xi;h)\in{}S^{m+m'-1}(\M)$.
\item
\begin{equation}\label{eq:SDOcommutator}
\sigma([A,B])=-ih\{\sigma(A),\sigma(B)\}+hr(x,\xi;h),
\end{equation}
where the error term $r(x,\xi;h)\in{}S^{m+m'-1}(\M)$ and $\{\cdot,\cdot\}$ denotes the Poisson bracket defined by
$$\{a,b\}=\frac{\partial a}{\partial x}\frac{\partial b}{\partial\xi}-\frac{\partial a}{\partial\xi}\frac{\partial b}{\partial x}.$$
\end{enumerate}

Note that like the standard calculus, error terms from commutation and composition are of lower order. In the semiclassical setting this means that each of these error terms is one factor of $h$ better than the first term. 

We now briefly discuss some standard properties and estimates from semiclassical analysis relevant to this paper.
\subsection{Elliptic operators}\label{sec:Elliptic}
We say that $a\in S(\R^n)$ is elliptic if $|a(x,\xi)|\ge c>0$ for all $(x,\xi)\in T^*\R^n$. If $a$ is elliptic, then there exists $b\in S(\R^n)$ such that $b(x,hD)$ is the inverse of $a(x,hD)$ modulo a smoothing operator. In fact, this assertion holds in the microlocal sense:
\begin{thm}[Inverses of elliptic \SDO s]\label{thm:elliptic}
Let $\chi\in S(\R^n)$. If $a\in S(\R^n)$ satisfies $|a(x,\xi)|\ge c>0$ for all $(x,\xi)\in\supp\chi$. Then there exists $b\in S(\R^n)$ such that
$$b(x,hD)a(x,hD)\chi(x,hD)=\chi(x,hD)+O_{L^2\to L^2}(h^\infty),$$
and
$$a(x,hD)b(x,hD)\chi(x,hD)=\chi(x,hD)+O_{L^2\to L^2}(h^\infty).$$
These equations also hold if we use Weyl quantisation instead of the standard one. Moreover, if $\chi\in C^\infty_0(T^*\R^n)$, then we can replace $O_{L^2\to L^2}(h^\infty)$ by $O_{\mathcal S'\to\mathcal S}(h^\infty)$.
\end{thm}

The proof of Theorem \ref{thm:elliptic} is similar to the analogous statement for the standard calculus. However, in this setting we express everything in powers of $h$. If $a(x,\xi;h)$ is bounded away from zero, then we may define $b_{0}(x,hD)$ by its symbol $b(x,\xi;h)=1/a(x,\xi;h)$. Since $a(x,\xi;h)$ is bounded away from zero $b_{0}(x,\xi;h)\in{}S(\R^{n})=S^{0}(\R^{n})$ Then the composition formula \eqref{eq:SDOproduct} gives that
$$\sigma(b_{0}(x,\xi;h)a(x,\xi;h))=1+r_{1}(x,\xi;h),$$
where $r_1(x,\xi;h)\in{}S^{m-1}(\R^{n})$, that is, it is one order of $h$ better than $a(x,\xi;h)$. Then we may set
$$b_{1}(x,\xi;h)=\frac{r_{1}(x,\xi;h)}{a(x,\xi;h)},$$ 
and the composition formula \eqref{eq:SDOproduct} gives
$$\left(b_{0}(x,hD)+hb_{1}(x,hD)\right)a(x,hD)=1+r_{2}(x,hD),$$
where $r_{2}(x,\xi;h)\in{}S^{m-2}(\R^{n})$. We can continue this process to produce
$$b_{N}(x,\xi;h)=\sum_{i=0}^{N}h^{i}b_{i}(x,\xi;h)$$
such that
$$b_{N}(x,hD)a(x,hD)=1+r_{N}(x,hD)$$
with $r_{N}(x,\xi;h)\in{}S^{m-N}$. Therefore modulo an $O(h^{\infty})$ term we may produce an inverse for $a(x,\xi;h)$.

The localisation property alone is enough to prove a range of $L^{p}$ estimates  (see e.g. Koch-Tataru-Zworski \cite[Lemma 2.2]{KTZ})
\begin{thm}[Semiclassical $L^p$ estimates]\label{thm:scLp} 
If $a\in S(\R)$, then
$$a(x,hD)=O\left(h^{n\left(\frac1q-\frac1p\right)}\right):L^p(\R^n)\to L^q(\R^n),$$
in which $1\le p\le q\le\infty$.
\end{thm}

An immediate consequence of Theorem \ref{thm:scLp} is that if $u=u(h)$ is a family of localised functions then (as we may locally treat $T^*\M$ as $\R^{2n}$)
$$\|u\|_q\le Ch^{n\left(\frac1q-\frac1p\right)}\|u\|_{L^p}+O(h^\infty),$$
for $1\le p\le q\le\infty$. 

\subsection{Quasimodes}\label{sec:qm}
Let $A(h)\in\Psi^m(\M)$. We define quasimodes as approximate solutions to $A(h)u(h)=0$. An $O_{L^{2}}(h^{\alpha})$ quasimode is a function $u\in{}L^{2}$ such that
$$\|A(h)u(h)\|_{L^2}\lesssim{}h^{\alpha}\|u(h)\|_{L^2}.$$
The finest quasimode resolvable by semiclasscical analysis is $O_{L^{2}}(h^{\infty})$ that is
$$\norm{A(h)u(h)}_{L^{2}}\lesssim{}h^{N}\norm{u}_{L^{2}}$$
for any $N$. 

We are particularly interested in the case when $\alpha=1$, that is, $O_{L^{2}}(h)$ quasimodes. These quasimodes behave well under localisation. Suppose $\chi(x,\xi)$ is a smooth compactly supported function then $\chi(x,\xi)\in{}S(\M)$ (we should think of this function as having small support in phase space). By the composition formula \eqref{eq:SDOproduct},
$$A(h)\chi(x,hD)u=a(x,hD)\chi(x,hD)u=\chi(x,hD)a(x,hD)u+r(x,hD)u,$$
where if $a(x,\xi;h)\in{}S^{m}(\M)$, then $r(x,\xi;h)\in{}S^{m-1}(\M)$. If $a(x,\xi;h)\in{}S(\M)$, which is often the case, then $r(x,\xi;h)\in{}S^{-1}(\M)$ and
$$\norm{r(x,hD)}_{L^{2}\to{}L^{2}}\lesssim{}h.$$
Therefore if $u$ is an $O_{L^{2}}(h)$ quasimode of $A(h)$, the function $v=\chi(x,hD)u$ is also an $O_{L^{2}}(h)$ quasimode of $A(h)$. This means that we may always decompose phase space into small size patches and work on each patch separately. Finer quasimode order is not necessary preserved under localisation making the treatment of such function technically difficult. Finally when $u$ is an $O_{L^{2}}(h)$ quasimode and $a(x,\xi;h)\in{}S(\M)$ we may ignore the dependence of the symbol on $h$. We write
$$a(x,\xi;h)=a_{0}(x,\xi)+h\tilde{a}(x,\xi;h),$$
where $a_{0}(x,\xi)=a(x,\xi,0)$ and $\tilde{a}\in{}S$. Therefore if $u$ is an $O_{L^{2}}(h)$ quasimode of $a(x,hD)$, then it is also an $O_{L^{2}}(h)$ quasimode of $a_{0}(x,hD)$. For this reason we only study quasimodes of operators associated to symbols $a(x,\xi)$ independent of $h$. 

\begin{ex}
Let $P(h)=-h^2\Delta-1$. Consider 
$$\Delta{}u_j=\lambda_{j}^{2}u_j,\quad\text{where }\norm{u_{j}}_{L^{2}}=1.$$
This family of eigenfunctions $\{u(h_{j})\}$ with $h_{j}=\lambda_{j}^{-1}$ is an quasimode of any order (i.e. $O_{L^{2}}(h^{\infty})$ quasimodes).

Now define the spectral clusters as linear combinations of eigenfunctions in a spectral window with fixed length:
$$u_\mu=\sum_{\lambda_j\in[\mu,\mu+1]}c_ju_j.$$
Then $u(h)=u_{1/h}$ is an $O_{L^2}$ quasimode to $P(h)$. The same conclusion also holds if one considers the smoothed spectral cluster (i.e. Sogge operator)
$$\chi_\mu u=\chi\left(\sqrt{-\Delta}-\mu\right)u\quad\text{with }\chi\in C^\infty_0(\R).$$
That is, $u(h)=\chi_{1/h}u$ is a family of $O_{L^2}$ quasimodes of $P(h)$. See Zworski \cite[\S7.4.1]{Zw} for more discussions on quasimodes. 
\end{ex}

It is often instructive to look at the singular structure of families of functions. This is usually done through the semiclassical wavefront set. The semiclassical wavefront set $\WF_h(u)$ of a tempered family $u(h)$ is the complement of the set of points $(x,\xi)\in T^*\M$ such that there exists $a\in C^\infty_0(T^*\M)$ with support sufficiently close to $(x,\xi)$ with
$$a(x,hD)u(h)=O_{L^2}(h^\infty).$$
Let $A\in\Psi^m(\M)$ with principal symbol $a\in S^m(\M)$ and $u(h)$ be a $O_{L^2}(h^\infty)$ family of quasimodes, i.e.
$$A(h)u(h)=O_{L^2}(h^\infty).$$
Then
$$\WF_h(u)\subset\{(x,\xi)\in T^*\M:a(x,\xi)=0\}.$$

\begin{ex}
The eigenfunctions (as a tempered family) $\{u_j\}_{j=0}^\infty=\{u(h_j)\}_{j=0}^\infty$ are the solutions to
$$P(h)u(h)=0,$$ 
where $P(h)=-h^2\Delta-1$ has symbol $p(x,\xi)=|\xi|_x^2-1$. Then the semiclassical wavefront set $\WF_h(u)$ is contained in $p^{-1}(0)=\{(x,\xi)\in T^*\M:|\xi|_x=1\}=S^*\M$, the cosphere bundle.
\end{ex}

As stated above we intend to work with localised symbols. To translate our results to eigenfunctions we need to know that they are indeed semiclassically localised. 

\begin{ex}
The eigenfunctions $\{u_j\}_{j=0}^\infty=\{u(h_j)\}_{j=0}^\infty$ admit localisation property: Note that $S^*\M$ is compact, one only needs to choose $\chi\in C^\infty_0(\M)$ such that $\supp\chi\supset S^*\M$ and equals $1$ around $S^*\M$, then
$$(1-\chi(x,h_jD))u(h_j)=O_{L^2}(h_j^\infty),$$
since $1-\chi\in C^\infty_0(\M)$ and $\supp(1-\chi)\cap\WF_h(u)=\emptyset$. 
\end{ex}

\subsection{Evolution equation and semiclassical Fourier integral operators}\label{sec:evo}
Let $P\in\Psi^m(\M)$ with principal symbol $p$. A key intuition in the study of quasimodes of $P$ is to link their behaviour to the classical flow given by
$$\begin{cases}
\dot{x}(t)=\partial_{\xi}p(x,\xi);\\
\dot{\xi}(t)=-\partial_{x}p(x,\xi).
\end{cases}$$
We often think of a quasimode as being comprised of small wave-packets tracking along trajectories of the classical flow. To make use of this intuition we need to understand semiclassical evolution equation. 

Consider the inhomogeneous semiclassical evolution equation.
$$\begin{cases}
(hD_t+A(t))u(t,x)=f(t,x),\quad\text{for }(t,x)\in\R\times\R^n,\\
u(0)=u_0.
\end{cases}$$
Here, $A(t)\in S(\R\times\R^n)$ is a family of \SDO s with $t$ as the parameter. The principal symbols $a(t,x,\xi)=\sigma(A(t))$ are real-valued and independent of $h$. We may use Duhamel's principle to reproduce $u$ in terms of the propagator $U(t)$ and the inhomogeneity. We represent $U(t)$ by an $h^{\infty}$-approximate propagator, that is, a solution to
\begin{equation}\begin{cases}
(hD_t+A(t))U(t)u=O_{L^2}(h^\infty),\quad\text{for }|t|\le T,\\
U(0)u=u.
\end{cases}\label{hinfprop}\end{equation}
We find a parametrix solution to \eqref{hinfprop} where $U(t)$ is a semiclassical Fourier integral operator and $u$ is semiclassically localised according to Definition \ref{localised}. Indeed for $T$ small, 
$$U(t)u(x)=\frac{1}{(2\pi h)^n}\int_{\R^n}\int_{\R^n}e^{\frac{i}{h}(\phi(t,x,\xi)-\langle{}y,\xi\rangle)}b(t,x,\xi;h)u(y)\,d\xi dy+E(t)u(x),$$
where $E(t)=O(h^\infty):\mathcal S'\to\mathcal S$, $\vp$ satisfies
$$\begin{cases}
\partial_t\phi(t,x,\eta)+a(t,x,\partial_x\phi(t,x,\xi))=0,\\
\phi(0,x,\xi)=\langle{}x,\xi\rangle,
\end{cases}$$
and
$$b(t,x,\xi;h)\in C^\infty_0(\R\times T^*\R^n\times\R).$$
\begin{proof}[Sketch of proof]
We seek a parametrix solution of the form
$$U(t)u(x)=\frac{1}{(2\pi{}h)^{n}}\int{}e^{\frac{i}{h}(\phi(t,x,\xi)-\langle{}y,\xi\rangle)}b(t,x,\xi;h)u(y)\,d\xi{}dy,$$
where
$$\phi(0,x,\xi)=\langle{}x,\xi\rangle\quad\text{and}\quad b(0,x,\xi;h)=1.$$
Note that this definitely satisfies the initial conditions. Then
$$hD_{t}[U(t)u(x)]=\frac{1}{(2\pi{}h)^{n}}\int{}e^{\frac{i}{h}(\phi(t,x,\xi)-\langle{}y,\xi\rangle)}\big[\phi_{t}(t,x,\xi)b(t,x,\xi;h)+hb_{t}(t,x,\xi)\big]u(y)\,d\xi{}dy,$$
and using the method of stationary phase
$$a(x,hD)[U(t)u(x)]=\frac{1}{(2\pi{}h)^{n}}\int{}e^{\frac{i}{h}\phi(t,x,\xi)-\langle{}y,\xi\rangle)}\big[a(t,x,\nabla_{x}\phi)b(t,x,\xi;h)+hr(t,x,\xi;h)\big]u(y)\,d\xi{}dy.$$
Therefore, if
$$\phi_{t}(t,x,\xi)+a(t,x,\nabla_{x}\phi)=0,$$
then 
$$\norm{(hD_{t}+A(t))U(t)}_{L^{2}\to L^2}\lesssim{}h.$$
To improve the error write
$$b(t,x,\xi;h)=1+\sum_{j=1}^{N}b_{j}(t,x,\xi;h).$$
Then by solving transport equations for each of the $b_{j}$ we may improve the error to $h^{N}$ for any $N$. Full details of the proof can be found in Zworski \cite[Chapter 10]{Zw}.
\end{proof}

\section*{Acknowledgements}
We want to thank Andrew Hassell for offering suggestions that helped to improve the presentation, Chris Sogge for pointing out the reference \cite{MSS}, and Nicholas Burq for bringing to our attention the relation between our $L^3$ bilinear estimate and the $L^2$ trilinear estimate in \cite{BGT5} . X.H. is partially supported by the Australian Research Council through Discovery Project DP120102019.

\end{document}